\documentclass{hha}

\disablequery
\usepackage{ifpdf}
\ifpdf
  \usepackage[pdftex]{hyperref}
\else
  \usepackage[hypertex]{hyperref}
\fi

\usepackage{amsrefs}
\usepackage[all]{xy}

\newdir{ >}{{}*!/-6pt/@{>}}

\newcommand{\id}{\mathrm{id}}

\newcommand{\N}{\mathbb{N}}
\newcommand{\op}{\mathrm{op}}
\newcommand{\pt}{\mathrm{pt}}
\newcommand{\ev}{\mathrm{ev}}

\newcommand{\colim}{\mathop{\mathrm{colim}}}

\newcommand{\obj}{\mathop{\mathrm{ob}}}

\newcommand{\unitaries}{\mathop{\mathit{uni}}}

\newcommand{\ism}{\mathop{\mathit{ism}}}

\newcommand{\R}{\mathbb{R}}
\newcommand{\C}{\mathbb{C}}
\newcommand{\F}{\mathbb F}

\newcommand{\inthom}{\mathrm{\underline{Hom}}}

\newcommand{\hilb}{\mathsf{Hilb}}
\newcommand{\sSet}{\mathsf{sSet}}
\newcommand{\Gpd}{\mathsf{Gpd}}
\newcommand{\Cat}{\mathsf{Cat}}

\newcommand{\Cstar}{\mathrm{C}^*\!}
\newcommand{\Cstarmax}{\mathrm{C}^*_{\max}}
\newcommand{\Cstarism}{\mathrm{C}^*_{\mathrm{ism}}}

\newcommand{\Cstarcatun}{\mathsf{C}_1^*\!\mathsf{cat}}
\newcommand{\Cstaralgun}{\mathsf{C}_1^*\!\mathsf{alg}}

\newcommand{\Weq}{\mathrm{Weq}}
\newcommand{\Fib}{\mathrm{Fib}}
\newcommand{\Cof}{\mathrm{Cof}}

\theoremstyle{hhadefinition}

\newtheorem{examples}[theorem]{Examples}
\newtheorem*{conventions*}{Conventions}

\theoremstyle{hhaplain}

\newtheorem{thm}[theorem]{Theorem}
\newtheorem{thm-defi}[theorem]{Theorem-Definition}
\newtheorem{prop}[theorem]{Proposition}
\newtheorem{cor}[theorem]{Corollary}
\newtheorem*{thm*}{Theorem}
\newtheorem*{lemma*}{Lemma}
\newtheorem*{cor*}{Corollary}

\theoremstyle{hharemark}

\newtheorem{construction}[theorem]{Construction}

\newtheorem{terminology}[theorem]{Terminology}

\newtheorem*{remark*}{Remark}

\begin{document}
\title{The unitary symmetric monoidal model category of small C*-categories}
\shorttitle{The unitary model of C*-categories}
\author{Ivo Dell'Ambrogio}
\email{Ivo.DellAmbrogio@math.univ-lille1.fr}
\address{%
Laboratoire Paul Painlev\'e,
Universit\'e de Lille 1,
Cit\'e Scientifique - B\^atiment M2,
F-59655 Villeneuve d'Ascq Cedex,
France}
\thanks{Research supported by the Stefano Franscini Fund, Swiss National Science Foundation grant Nr.~PBEZP2-125724.}


\begin{abstract}
We produce a cofibrantly generated simplicial symmetric monoidal model structure for the category of (small unital) $\Cstar$-categories, whose weak equivalences are the unitary equivalences. The closed monoidal structure consists of the maximal tensor product, which generalizes that of $\Cstar$-algebras, together with the Ghez-Lima-Roberts $\Cstar$-categories of $*$-functors, $\Cstar(A,B)$, providing the internal Hom's. 
\end{abstract}

\received{March 14, 2011}   
\revised{June 13, 2012}    
\published{Month Day, Year}  
\submitted{Charles A. Weibel}  
\volumeyear{2012} 
\volumenumber{14} 
\issuenumber{2}   
\startpage{1}     

\maketitle

\vspace{1em}
\begin{center}
{\it To Tamaz Kandelaki.}
\end{center}


\section{Introduction}

In the same way that $\Cstar$-algebras provide an axiomatization of precisely the norm-closed and $*$-closed subalgebras $A\subseteq \mathcal L(H)$ of bounded operators on Hilbert space, $\Cstar$-categories axiomatize the norm-closed and $*$-closed sub\emph{categories} of the category of all Hilbert spaces and bounded operators between them. An important abstract example is the collection of all $*$-homomorphisms between any two fixed $\Cstar$-algebras (or, more generally, of all $*$-functors between two $\Cstar$-categories) $A$ and~$B$, which forms a $\Cstar$-category $\Cstar(A,B)$ in a natural way. 
Another example is the $\Cstar$-category of Hilbert modules over a fixed $\Cstar$-algebra. $\Cstar$-categories occur naturally in the detailed study of analytic assembly maps and 
they provide functorial versions of various constructions in coarse geometry 
(see the numerous references provided in~\cite{mitchener:Cstar_cats}).
Most notably, perhaps, $\Cstar$-categories (with additional structure) feature prominently in Doplicher-Roberts duality of compact groups~\cite{doplicher-roberts:new_duality}.
 References to other domains of applications are given in~\cite{glr,vasselli:bundlesII}. At the end of this introduction, and in view of our own `homotopical' results, we shall advocate an entire new course of application.

The abstract theory of $\Cstar$-categories was initiated in the article~\cite{glr} by Ghez, Lima and Roberts, who were mainly concerned with W*-categories (generalizing von Neumann algebras and comprizing the representation $\Cstar$-categories $\Cstar(A,\hilb)$). It was later picked up by Mitchener~\cite{mitchener:Cstar_cats}, who described useful basic constructions such as the minimal tensor product $\otimes_{\min}$ and the reduced and maximal groupoid $\Cstar$-categories $\mathrm C^*_r\mathcal G$ and $\Cstarmax \mathcal G$ associated to every discrete groupoid~$\mathcal G$; 
by Kandelaki~\cite{kandelaki:multiplier}, who has introduced multiplier $\Cstar$-categories and has used them to characterize $\Cstar$-categories of countably generated Hilbert modules; 
and by Vasselli~\cite{vasselli:bundles, vasselli:bundlesII}, who also studies multiplier $\Cstar$-categories as well as bundles of $\Cstar$-categories in order to extend Doplicher-Roberts duality to groupoids.
All three \Changed{latter} authors have worked with non-unital $\Cstar$-categories, which appear as the kernels of $*$-functors. 

In the present article, we extend the basic theory by generalizing some other constructions from the world of algebras to that of categories. 
For reasons that will become apparent, we limit ourselves to unital $\Cstar$-categories (\emph{i.e.}, those with identity arrows for all objects) and units preserving functors.
We begin with some recollections (Section~1), and proceed (Section~2) with a concise, but precise, treatment of so-called ``universal constructions'', which include as special cases all small colimits in the category of small $\Cstar$-categories and Mitchener's maximal groupoid $\Cstar$-categories. Another, apparently new, example is provided by the maximal tensor product $A \otimes_{\max} B$ of two $\Cstar$-categories. 
Its theoretical importance is testified by the following result (see Theorem~\ref{thm:closed_monoidal_Cstarcatun}):

\begin{thm*}
The maximal tensor product $\otimes_{\max}$ defines a closed symmetric monoidal structure on the category $\Cstarcatun$ of small (unital) $\Cstar$-categories and (units preserving) $*$-functors, whose internal Hom objects are precisely the $\Cstar(A,B)$.
\end{thm*}

We observe (Theorem~\ref{thm:max_products}):
\begin{thm*}
The maximal groupoid $\Cstar$-category is a symmetric monoidal functor
\[
\Cstarmax(\mathcal G_1 \times \mathcal G_2)\simeq \Cstarmax (\mathcal G_1) \otimes_{\max} \Cstarmax (\mathcal G_2)
\]
from discrete groupoids to $\Cstar$-categories.
\end{thm*}

Our main results are to be found in Section~3, where we uncover some remarkable features of $\Cstar$-categories. We collect them here in the following omnibus theorem, 
comprising Theorem~\ref{thm:first_model}, Proposition~\ref{prop:cof_gen} and Theorems~\ref{thm:simpl_struct} and~\ref{thm:algebra}.

\begin{thm*}
The category $\Cstarcatun$ of small $\Cstar$-categories and $*$-functors admits the structure of a cofibrantly generated simplicial model category, where:
\begin{enumerate}
\item Weak equivalences are the unitary equivalences, namely those $*$-functors\break $F\colon A\to B$ that are equivalences of the underlying categories, such that the isomorphisms $FF^{-1}\simeq \id_B$ and $F^{-1}F\simeq \id_A$ can be chosen with unitary components. \textup(In fact, as it turns out, the latter property is automatic; see Lemma~\ref{lemma:un_eq}.\textup)
\item The cofibrations are the $*$-functors that are injective on objects.
\item The fibrations are the $*$-functors $F\colon A\to B$ that allow the lifting of unitary isomorphisms $Fx \simeq y$.
\end{enumerate}
Moreover, the above \emph{unitary model structure} is compatible with the maximal tensor product~$\otimes_{\max}$, that is, together they endow $\Cstarcatun$ with the structure of a symmetric monoidal model category. In fact, the previous theorem shows that $\Cstarcatun$ is a symmetric $\sSet$-algebra in the sense of Hovey~\cite{hovey:model}.
\end{thm*}

We should note that the simplicial structure is defined via the nice formulas
\[
A\otimes K := A\otimes_{\max} \pi K
\quad, \quad
A^K := \Cstar(\pi K, A)
\quad\textrm{ and } \quad
\mathrm{Map}(A,B):= \nu \Cstar(A,B)
\]
 (for $A,B\in \Cstarcatun$ and $K \in \sSet$), where $\pi K= \Cstarmax(\Pi K)$ is the groupoid $\Cstar$-category of the fundamental groupoid of~$K$, and $\nu$ denotes the simplicial nerve of the subcategory of unitary isomorphisms.
 
The whole structure is intimately connected to the canonical model on the category of small categories (or small groupoids) of homotopists' folklore, where weak equivalences are the categorical equivalences in the usual sense.  In the last subsection of our article we explain how these models are related to each other. 

We now briefly suggest why the operator algebraists should care.

The hard-won experience of representation theorists and algebraic geometers has shown that a good way of studying certain invariants of rings, dg-algebras, or schemes, such as $K$-theory and Hochschild and cyclic (co)cohomology, is to proceed as follows (see~\cite{keller_dg}): 
First, substitute your object (ring, scheme,~\ldots) with a suitable \emph{small dg-category} of representations. The setting of dg-categories is convenient in part because it is closed under many constructions, such as taking tensor products, functor categories, localizations~\ldots (only the latter works well for, say,  triangulated categories). 
Second, factor your invariant through a suitable localization of the category of small dg-categories, by inverting classes of  morphisms (quasi-equivalences, Morita equivalences) that
induce isomorphisms of the invariants in question. 
In order to retain control on the result, one should realize this localization as the homotopy category of a suitable model structure on dg-categories. Another reason dg-categories are convenient is that they allow such models. The resulting localization is now a unified convenient setting where the powerful methods of modern homotopy theory can be applied to the collective study of the invariants.

We propose that a similar strategy could be useful for enhancing the study of invariants of $\Cstar$-algebras (groupoids,~\ldots) such as (equivariant and bivariant) $K$-theory. We have provided here the first pieces of the puzzle: the category of small $\Cstar$-categories is closed under numerous constructions, and carries a nice model structure for what is perhaps the strongest and most natural notion of equivalence after that of isomorphism, namely, unitary equivalence. The next logical step should be to investigate Morita(-Rieffel) equivalence by providing a suitably localized model structure, and to identify convenient small $\Cstar$-categories of representations  (Hilbert modules, \ldots) that should stand for the $\Cstar$-algebra. 

If this vision can be carried out to some extent, we prophesy bountiful applications.


 \begin{conventions*}
The base field will be denoted by~$\F$, and it is either the field $\R$ of real numbers or $\C$ of complex numbers. In this article, all categories (included $\Cstar$-categories) have identity arrows $1_x$ for all objects~$x$, and all functors (included $*$-functors) are required to preserve the identities (but \emph{cf.\ }Remark~\ref{rem:non_unital}).
\end{conventions*}

\ack
I am grateful to Mart Abel, Theo B\"uhler, Tamaz Kandelaki, Ralf Meyer, Ansgar Schneider and Gon\c calo Tabuada for their interest and helpful remarks. In addition, I would like to thank the anonymous referee for some useful comments.


\section{Recollections on $\Cstar$-categories}

\subsection{Definitions and examples} 

An \emph{$\F$-category} is a category enriched over $\F$-vector spaces. Concretely, an $\F$-category consists of a category $A$ whose Hom sets $A(x,y)$ have the structure of $\F$-vector spaces, and whose composition law consists of linear maps $A(y,z)\otimes_{\F} A(x,y) \to A(x,z)$.

A \emph{$*$-category} $ A$ is an $\F$-category equipped with an \emph{involution}, by which we mean an involutive antilinear contravariant endofunctor which is the identity on objects. In detail,  an involution consists of a collection of maps ${A}(x,y)\to{A}(y,x)$ for all objects $x,y\in \obj A$, each denoted by $a \mapsto a^*$, satisfying the identities 
 \begin{itemize}
 \item[(i)] $(za+wb)^*=\overline z a^*+\overline wb^*$ for $z,w\in \F$ and parallel arrows $a,b\in A$,
 \item[(ii)] $(ba)^*=a^*b^*$ for $a$ and $b$ composable,
 \item[(iii)]  ${(a^*)}^*=a$ for all arrows $a$.
 \end{itemize}
A \emph{$*$-functor} $F\colon A\to B$ between $*$-categories is an $\F$-linear functor  commuting with the involution: 
$F(a^*)=F(a)^*$.  In the following we shall occasionally work in the purely algebraic category of small $*$-categories and $*$-functors. 

A \emph{\textup(semi-\textup)normed category} is an $\F$-category whose Hom spaces  $A(x,y)$ are (semi-) normed in such a way that the composition is sub\-multi\-pli\-cat\-ive: 
\[
\|b \cdot a\| \leq\|b\|  \|a\|
\]
for any two composable arrows $b\in A(y,z)$ and $a\in A(x,y)$. A normed category is \emph{complete}, or is a \emph{Banach category}, if each  Hom space is complete.
Every semi-normed category $ A$ can be completed to a Banach  category, by first killing the null ideal
 $ \{a \mid \|a\| =0 \}\subseteq  A$ and then completing each quotient Hom space.

 We now come to our main object of study:
 
 \begin{definition}[$\Cstar$-category] 
 \label{defi:Cstar_cat}
 A \emph{pre-$\Cstar$-category} is a normed $*$-category  $ A$ satisfying the two additional axioms:
 \begin{itemize}
 \item[(iv)] \textbf{$\Cstar$-identity:} 
 $\|a^*a\|=\|a\|^2$ for all arrows $a\in A$.
 \item[(v)] \textbf{Positivity:} 
 For every arrow $a\in A(x,y)$, the element $a^*a$ of the endomorphism algebra $A(x,x)$ is positive (\emph{i.e.}, its spectrum $\{z\in \F \mid a^*a- z1_x$ is not invertible in $A(x,x)\}$ is contained in $[0,\infty[\,$). 
 \end{itemize}
A $\Cstar$-\emph{category} is a complete pre-$\Cstar$-category.  
We shall denote by $\Cstarcatun$ the category of small $\Cstar$-categories and $*$-functors between them.
 \end{definition}
 
 \begin{remark}
Axiom (v) is not particularly elegant. We refer to~\cite{mitchener:Cstar_cats}*{$\S$2} for a discussion of alternative axiomatizations of $\Cstar$-categories in the real and complex cases, as well as for a simple example showing that (v) is necessary (\emph{i.e.}, it does not follow from the other axioms). 
In the complex case, we may substitute (v) with the following axiom:
\begin{itemize}
\item[(v)$'$] For every $a\in A(x,y)$ there is a $b\in A(x,x)$ such that $a^*a=b^*b$.
\end{itemize}
In both the real and complex cases, we may simultaneously substitute (v) with~(v)$'$, and (iv) with:
\begin{itemize}
\item[(iv)$'$] \textbf{$\Cstar$-inequality:} $\|a\|^2\leq \|a^*a+b^*b\|$ for all parallel arrows $a,b\in A(x,y)$.
\end{itemize} 
In any case, the ultimate justification for the notion of $\Cstar$-category is that it precisely captures norm-closed, $*$-closed subcategories of the category of Hilbert spaces and bounded operators, by Proposition~\ref{prop:GNS}.
 \end{remark}

 \begin{examples}\label{ex:Cstarcats}
We mention a few examples that are relevant to this article.

 \begin{itemize}
\item[(a)]
 $\Cstar$-categories with one single object can be identified with (unital) $\Cstar$-algebras, and $*$-functors between them coincide with the usual (unit preserving) $*$-homo\-morph\-isms. Conversely, the endomorphism ring $A(x,x)$ at each object $x$ of a $\Cstar$-category $A$ is clearly a $\Cstar$-algebra.

\item[(b)] 
For instance, the base field $\F$ can be considered as a $\Cstar$-category with one single object, say~$\bullet$,  with endomorphism algebra $\F(\bullet,\bullet)= \F\, 1_{\bullet}$. For any $\Cstar$-category~$A$, the $*$-functors $F\colon\F\to A$ are in natural bijection with the objects $x=F(\bullet)$ of~$A$.
 
\item[(c)]
 It is easily verified (\emph{cf.\ e.g.}, \cite{mitchener:Cstar_cats}*{Prop.\ 5.1}) that the category $\hilb$ of Hilbert spaces and bounded linear operators is a (non small) $\Cstar$-category, for the usual operator norm and adjoints. A $*$-functor $F\colon A\to \hilb$ is called a \emph{representation of}~$A$. In Section~\ref{sec:univ} we shall consider representations $F\colon A\to \hilb$ of general $*$-categories~$A$, and even of ``quivers with admissible relations''.
 
 \item[(d)] If $A$ is a $\Cstar$-category, an $\F$-linear subcategory $B\subseteq A$ whose every Hom space is closed with respect to norm and involution is again a $\Cstar$-category, with the operations induced from~$A$. 
 (The only dubious point is the positivity axiom; but the spectrum, and thus the positivity, of an operator is independent from the ambient $\Cstar$-algebra in which it is computed, \emph{cf.}~\cite{black1}*{II.1.6.7}.)
 The sub-$\Cstar$-categories of $\hilb$ are called \emph{concrete $\Cstar$-categories}.
 
\item[(e)] 
A motivating example of a $\Cstar$-category with many objects is the \emph{internal Hom} $\Cstar(A,B)$, defined for any two $\Cstar$-categories~$A,B$. 
Its objects are the $*$-functors $F\colon A\to B$, and morphisms $\alpha\colon F\to F'$ are bounded natural transformations (see \S\ref{subsec:Cstarhom} below). 
Note that, even when $A$ and~$B$ are $\Cstar$-algebras, $\Cstar(A,B)$ is hardly ever a $\Cstar$-algebra. 
 

\end{itemize}
\end{examples}

As for $\Cstar$-algebras, we say that an arrow $a\in A(x,y)$ in a $\Cstar$-category is \emph{unitary} if its adjoint is its inverse:
 $a^*a=1_x\in A(x,x)$ and $aa^*=1_y\in A(y,y)$. The present article could be construed as an investigation of the following notion:

\begin{definition}[Unitary equivalence] 
\label{defi:unitary_eq}
A $*$-functor $F\colon A\to B$ is a \emph{unitary equivalence} if there exist a $*$-functor $G\colon B\to A$ and isomorphisms $u\colon GF\simeq \id_{A} $ and $v\colon FG\simeq  \id_{B} $ such that the components $u_x\in A(GFx,x)$  and $v_y\in B(FGy,y)$ are unitary elements for all $x\in \obj A$ and $y\in \obj B$.
\end{definition}

Unitary equivalences are simply called \emph{equivalences} in~\cite{mitchener:Cstar_cats}.  

\begin{remark}
If $A$ and $B$ are W*-categories then according to~\cite{glr}*{after Def.\ 6.3} every $*$-functor $A\to B$ which is an equivalence is also a unitary equivalence, as can be seen by objectwise applying  the polarization identity. Although the latter is not available in general $\Cstar$-categories (or even $\Cstar$-algebras), the conclusion is actually true for general $\Cstar$-categories, as we have learned from~\cite{kandelaki:KK_K}*{\S3.1}. 
As this fact had caused us some confusion in the past we now record it carefully in the next proposition.
\end{remark}

\begin{prop}
\label{prop:uni_eq}
If two objects of a $\Cstar$-category $A$ are isomorphic, then they are also isomorphic via a unitary isomorphism.
Indeed, if $a\colon x\to y$ is the given isomorphism, a unitary isomorphism $u\colon x\to y$ is obtained by the formula $u:=a (a^*a)^{-1/2}$.
\end{prop}

\begin{proof}
(\emph{Cf.}~\cite{blackadar:Kth_op_alg}*{Prop.\ 4.6.4}.)
 If $a\in A(x,y)$ is an isomorphism, it follows from  $1_x^*=1_x$ and $1_y^*=1_y$ that $a^*$ is also invertible with inverse $(a^{-1})^*$.
 Thus $a^*a$ is an invertible element of the endomorphism $\Cstar$-algebra $A(x,x)$, and is moreover a positive element by axiom (v) of $\Cstar$-categories. By functional calculus there exists a self-adjoint $r:=(a^*a)^{-1/2}\in A(x,x)$ commuting with $a^*a$ and satisfying $r^2=(a^*a)^{-1}$.
Setting $u:=ar$, we can now compute 
$u^*u= r^*a^*ar= r^2 a^*a = 1_x$ and
$uu^* = ar^2a^* = a (a^*a)^{-1}a^*=aa^{-1}(a^*)^{-1}a^*=1_y$, proving the claim.
\end{proof}
 
It follows in particular that a unitary equivalence between two $\Cstar$-categories is precisely the same thing as a $*$-functor which induces an equivalence of the underlying ordinary categories (see Lemma~\ref{lemma:un_eq} if necessary). 

\subsection{Some basic properties}
\label{subsec:basics}

Many pleasant features of $\Cstar$-algebras generalize almost effortlessly to $\Cstar$-categories (see~\cite{glr, mitchener:Cstar_cats}). We now recall those that we shall use in the following, often without mention.  

First of all, every $*$-functor $F\colon A\to B$ between $\Cstar$-categories is automatically norm decreasing:
\[
\|F(a)\|_B \leq \|a\|_A 
\textrm{ for all } a\in A,
\]
so in particular it is continuous on every Hom space. If $F$ is \emph{faithful} (that is, injective on arrows) then it is automatically isometric. By combining these two facts, we see that every $*$-category possesses at most one \emph{$\Cstar$-norm} (\emph{i.e.}, one that turns it into a $\Cstar$-category). Similarly, we see that an invertible $*$-functor $F\colon A\to B$ (an isomorphism in $\Cstarcatun$) is isometric and thus identifies \emph{all} the structure of the $\Cstar$-categories $A$ and~$B$, norms included. More generally, every unitary equivalence $F$ is also isometric, since 
$
\|a\| \geq \|Fa\| \geq \|GFa\|= \|u_y^* \cdot a \cdot u_x\| = \|a\|
$
(for any $a\in A(x,y)$, where $G$ is a $*$-functor quasi-inverse to $F$ and $u\colon GF\simeq \id$ is unitary).

In the next section we shall give a unified treatment of colimits in $\Cstarcatun$ together with other universal constructions. The case of limits is much easier:

\begin{lemma} \label{lemma:complete}
The category $\Cstarcatun$ has all small limits.
\end{lemma} 

\begin{proof}
As in any category, the existence of arbitrary limits follows from that of arbitrary products and equalizers. For the latter, let $F,G\colon A\rightrightarrows B$ be two $*$-functors.
Define a new $\Cstar$-category $E$ to be the $\F$-linear subcategory of $A$ with objects
$\obj E:=\{x\in \obj A\mid Fx=Gx \}$ 
and maps $E(x,y):=\{a\in A(x,y)\mid F(a)=G(a)\}$.
It follows immediately from the properties of $F$ and $G$ that the $E(x,y)$ are norm-closed ($*$-functors are continuous!) and $*$-closed subspaces containing the identities and closed under composition; therefore $E$ is a $\Cstar$-category with the induced operations of~$A$ (Ex.~\ref{ex:Cstarcats}\,(d)). 
It is now straightforward to check that the inclusion $*$-functor $E\hookrightarrow A$ is the equalizer of $F$ and~$G$.
The product of a set $\{A_i\}_i$ of $\Cstar$-categories has also an easy construction: let $P$ be the category with objects $\obj P:= \prod_i \obj A_i$ and with morphism sets 
\[
P \big( (x_i) , (y_i) \big):= 
\left\{ (a_i)   \bigg|    \| (a_i) \|_{\infty} := \sup_i \|a_i\|_{A_i} < \infty \right\} 
\subseteq \prod_i A_i  (x_i, y_i).
\] 
As with $\Cstar$-algebras, one checks easily that the coordinatewise operations and the norm $\|\cdot \|_{\infty}$ equip $P$ with the structure of a pre-$\Cstar$-category.
Moreover, each $P(x,y)$ is complete, so $P$ is a $\Cstar$-category. This~$P$, together with the canonical projections $P\to A_i$, provides the product of the $A_i$ in~$\Cstarcatun$.
\end{proof}

Finally, one can always prove properties of $\Cstar$-categories by way of the basic representation theorem: 

\begin{prop} \label{prop:GNS}
Every small $\Cstar$-category is isomorphic to a concrete $\Cstar$-cat\-eg\-ory. Or equivalently, every small $\Cstar$-category $A$ has a faithful representation $A\to \hilb$. 
\end{prop}

\begin{proof}
See~\cite{glr}*{Prop.\,1.9},~\cite{micic}*{Thm.\,3.48} and~\cite{mitchener:Cstar_cats}*{Thm.\,6.12} for more details.
\end{proof}

\subsection{$\Cstar$-categories of $*$-functors}
\label{subsec:Cstarhom}

Let $A$ and $B$ be two $\Cstar$-categories, with $A$ small. We recall from~\cite{glr} the definition of the $\Cstar$-category $\Cstar(A,B)$ (denoted $(\mathfrak A,\mathfrak B)$ in \emph{loc.\,cit.}).
Its object set consists of all $*$-functors $F\colon A\to B$. For any two $*$-functors $F,F'\in \obj \Cstar(A,B)$, the set of natural transformations $\alpha\colon F\to F'$ between them forms an $\F$-vector space.  Setting
\[
\| \alpha \|_{\infty} := \sup_{x\in \obj A} \| \alpha_x \|_B ,
\]
we define the Hom spaces of $\Cstar(A,B)$ to be the subspaces
\[
\Cstar(A,B)(F,F') := \{ \alpha\colon F\to F' \mid \| \alpha \|_{\infty} < \infty \} 
\] 
of \emph{bounded} natural transformations.
Then the pointwise algebraic operations
\[
(\alpha+ z \alpha')_x = \alpha_x + z \alpha'_x 
\quad,\quad
(\alpha \cdot \alpha')_x = \alpha_x \cdot \alpha'_x
\quad,\quad
(\alpha^*)_x = (\alpha_x)^*
\quad,
\]
together with the norm $\|\cdot \|_{\infty}$, turn 
$\Cstar(A,B)$ into a $\Cstar$-category (\cite{glr}*{Prop.\,1.11}). 

\begin{remark}
Not all morphisms of $*$-functors are bounded. 
For instance, let $A$ be the $\Cstar$-algebra with $\obj A := \N$ and with Hom spaces $A(n,n):=\F$ and $A(n,m):=0$ for~$n\neq m$ (this is the coproduct $A=\coprod_{\N} \F$). 
Then the collection $(n)_{n\in \N}$ defines a natural transformation $\alpha\colon \id_A \to \id_A$, but obviously $\|\alpha \|_{\infty}= \infty$.
\end{remark}

It is natural to ask  whether there exists a monoidal structure on $\Cstarcatun$ for which $\Cstar(-,-)$ is the internal Hom. The answer is Yes, and will be provided by the maximal tensor product in Theorem~\ref{thm:closed_monoidal_Cstarcatun}.

\section{Universal $\Cstar$-categories}
\label{sec:univ}

In this section we construct colimits in $\Cstarcatun$, groupoid $\Cstar$-algebras $\Cstarmax\mathcal G$ (see also their generalizations $\Cstarism(C)$ in Definition~\ref{defi:ism}) and maximal tensor products $A \otimes_{\max} B$. These are all special cases of  ``universal constructions'', whose analog for $\Cstar$-algebras is well-known and widely used. 
For future use, we provide here a treatment of such general universal constructions for $\Cstar$-categories. This also suits the spirit of the article, which is to show that $\Cstar$-categories offer a powerful and flexible alternative setting to that of $\Cstar$-algebras.

\subsection{$\Cstar$-categories defined by generators and relations}

 As for $\Cstar$-algebras or pro-$\Cstar$-algebras~\cite{phillips}, it is possible to construct a $\Cstar$-category by generators and relations, provided the set of relations is well-behaved. 
Here in addition we have to specify a set of objects, and for this end it is natural to employ the notion of quiver.
A \emph{quiver} is a labeled oriented graph, possibly with loops and multiple edges; in other words, a quiver~$Q$ consists of a set $\obj Q$ of vertices, here called \emph{objects}, and of a set $Q$ of \emph{arrows}, together with \emph{source} and \emph{target} functions $s,t\colon Q\to \obj Q$. We write $q\in Q(x,y)$ to indicate that the arrow $q$ has source $s(q)=x$ and target $t(q)=y$.
A \emph{morphism of quivers} $\rho\colon Q\to Q'$ assigns to every object $x\in \obj Q$ an object $\rho(x)\in \obj Q'$, and to every arrow $q\in Q(x,y)$ an arrow $\rho(q)\in Q'(\rho(x),\rho(y))$.
Thus categories are quivers with the extra structure given by composition, and functors are quiver morphisms that preserve composition.

We also want to specify relations that should hold between the arrows, when interpreted in a $\Cstar$-category. 
Given a quiver~$Q$, a \emph{relation for~$Q$} is simply a statement about arrows of~$Q$ which makes sense for elements of a $\Cstar$-category. 
A \emph{quiver with relations} $(Q,R)$ is a quiver~$Q$ together with a set~$R$ of relations for~$Q$. 
A \emph{representation}  of the quiver with relations $(Q,R)$ is a quiver morphism $\rho\colon Q\to A$ into (the underlying quiver of) a $\Cstar$-category~$A$, such that the arrows $\rho(q)$ satisfy in~$A$ all the relation of~$R$. 
(For example, a relation $r\in R$ for $Q$ may read ``$ \|q_1^{1/3}\|\leq \|q_2\cdot q_3^* -1_{s(q_1)}\| $'', for some arrows $q_1,q_2,q_3\in Q$. If $\rho\colon Q\to A$ is a representation of~$(Q,R)$, then the statement ``$\|\rho(q_1)^{1/3}\|\leq \| \rho(q_2)\cdot \rho(q_3)^* -1_{s(\rho(q_1))}\|$'' is required to hold in~$A$. In particular, it must make sense.)

\begin{remark}  \label{rem:non_unital}
The only relations that will be used in this article are statements involving algebraic combinations of arrows and norms thereof.
However, since there is no extra effort involved, it seems worth it to give here the general treatment of universal constructions in view of future applications, where more sophisticated operations and conditions could be involved. 
We are thinking for instance of analytic functions such as exponentiation, or continuity conditions when dealing with topological spaces (\emph{cf.}\,\cite{phillips} and~\cite{joachim-johnson}).
We also note at this point that the definitions and results of this section all have evident non-unital versions. 
But for simplicity and focus we shall stick with unital $\Cstar$-categories.
\end{remark}

The next definition, as well as the proof of Theorem~\ref{thm:univ}, have been adapted from Phillips~\cite{phillips}*{Def.\,1.3.1,\,1.3.2}.

\begin{definition}[Admissible quiver with relations] 
\label{defi:admissibleQR} 
A quiver with relations $(Q,R)$ is \emph{admissible}, if the following five requirements are satisfied.
\begin{itemize}
\item[(1)] The unique quiver morphism $\rho\colon Q\to \mathbf 0$ to the final $\Cstar$-category ($=$ the zero $\Cstar$-algebra) is a representation of $(Q,R)$. 
\item[(2)] If $\rho\colon Q\to A$ is a representation of $(Q,R)$ and if $B$ is a sub-$\Cstar$-category of~$A$ which contains the image under~$\rho$ of all objects and arrows of~$Q$, then the restriction $\rho'\colon Q\to B$ is also a representation of~$(Q,R)$. 

\item[(3)]  If $\rho\colon Q\to A$ is a representation of $(Q,R)$ and if $\varphi\colon A\to B$ is an isomorphism of $\Cstar$-categories, then the composition $\varphi \cdot \rho\colon Q\to B$ is also a representation of~$(Q,R)$.

\item[(4)] For every arrow $q\in Q$ there exists a constant $c(q)$ such that $\|\rho(q)\|\leq c(q)$ for all representations $\rho$ of $(Q,R)$.

\item[(5)] If $\{\rho_{i} \colon Q\to A_i \}_{i}$ is a nonempty small set of representations of~$(Q,R)$, then the quiver morphism $(\rho_i)_i \colon Q \to \prod_i A_i$ into the product $\Cstar$-category mapping $q \in Q(x,y)$ to $(\rho_i(q) )_i$
(which is always well defined if (4) holds, \emph{cf.}\ the proof of Lemma~\ref{lemma:complete}) is again a representation of $(Q,R)$.

\end{itemize}
\end{definition}

\begin{thm} \label{thm:univ}
Given any small \textup(\emph{i.e.}, the arrow and object sets are small\textup)  admissible quiver with relations~$(Q,R)$, there exist a small $\Cstar$-category $\mathrm{U}(Q,R)$ and a representation $\rho_{(Q,R)}\colon Q\to \mathrm{U}(Q,R)$ of $(Q,R)$ such that, for every representation $\rho\colon Q\to A$ of $(Q,R)$ into some \textup(possibly large\textup) $\Cstar$-category, there exists a unique $*$-functor $\tilde \rho:\mathrm{U}(Q,R)\to A$ such  that $\rho=\tilde \rho \cdot \rho_{(Q,R)}$. 
\begin{equation*}
\xymatrix{
Q \ar[rr]^-{\rho_{(Q,R)}} \ar[dr]_{\rho} && \mathrm{U}(Q,R) \ar@{..>}[dl]^{\tilde \rho} \\
&A&
}
\end{equation*}
We call $\mathrm{U}(Q,R)$ the \emph{universal $\Cstar$-category generated by $Q$ with relations~$R$}.
\end{thm}

\begin{construction}
\label{constr:free}
Given any quiver~$Q$, the \emph{free $*$-cat\-eg\-ory over~$Q$}  is the $*$-category $\mathrm F(Q)$ with the same set of objects $\obj Q$ and whose morphisms are recursively constructed by adding formal units, compositions, adjoints and finite linear combinations of arrows of~$Q$, and then by modding out the algebraic relations which make the units units, the Hom sets $\F$-linear spaces, composition bilinear and $*$ an involution. 
This construction provides a left adjoint for the forgetful functor from $*$-categories to quivers, whose unit is the evident inclusion $Q\to \mathrm F(Q)$. 
Indeed, every quiver morphism $\rho\colon Q\to A$ to a $*$-category extends to a unique $*$-functor $\mathrm F(Q)\to A$, by setting
$\rho(1_x):= 1_{\rho(x)}$ and
\[
\rho(q \cdot p):=\rho(q)\cdot \rho (p)
\quad,\quad
 \rho \left(\sum_i z_i q_i \right):=\sum_i z_i\rho(q_i)
\quad,\quad \rho(q^*):=\rho(q)^*
 \]
for each of the three moves in the recursive definition of the arrows of~$\mathrm F(Q)$.

\end{construction}

\begin{terminology}
We say that a representation $\rho\colon Q\to A$ is \emph{dense} if the smallest sub-$\Cstar$-category of $A$ containing the image $\rho(Q)$ is $A$ itself. If $\rho\colon Q\to A$ is a dense representation, it is clearly surjective on objects. Moreover, every $*$-functor $A\to B$ to another $\Cstar$-category, being continuous, is uniquely determined by the image of~$\rho (Q)$. 
Two representations $\rho\colon Q\to A$ and $\rho'\colon Q\to A'$ of~$Q$ are \emph{isomorphic} if there exists an isomorphism $\varphi\colon A\simeq A'$ of $\Cstar$-categories such that~$\varphi\rho=\rho'$.
\end{terminology}

\begin{lemma}
\label{lemma:smalll}
The isomorphism classes of dense representations of~$(Q,R)$ form a small set.
\end{lemma}

\begin{proof}
Every representation $\rho\colon Q\to A$ factors uniquely as $Q\to \mathrm F(Q)\to A$, where the second map $\pi\colon\mathrm F(Q)\to A$ is the induced $*$-functor on the free $*$-category over~$Q$, as in Construction~\ref{constr:free}. 
If $\rho$ is dense, then $\pi$ has dense image in each Hom set (because the norm-closure of the image is a sub-$\Cstar$-category of~$A$ still containing the image). 
Moreover, in this case $A$ is isomorphic to the completion of the quotient $*$-category $\mathrm F(Q)/\ker(\pi)$ with respect to the norm induced from~$A$. 
But $Q$ is small, so there is only a set of possible quotients $\mathrm F(Q)/I$ (where $I$ indicates some $\F$-linear categorical ideal compatible with the involution), and for each of them there is only a set of possible norms~$\nu$. Using condition~(3), we conclude that every dense representation is isomorphic to a representation of the form $Q\to \smash{\overline{\mathrm F(Q)/I}^{\,\nu}}$ just described, of which there is only a set.
\end{proof}

\begin{proof}[Proof of Theorem~\ref{thm:univ}.]
Given an admissible quiver with relations~$(Q,R)$, 
choose a small full set  of representatives $\{\rho_i\colon Q\to A_i\}_{i\in I}$ for the isomorphism classes of its dense representations, as in Lemma~\ref{lemma:smalll}. 
By condition~(1), $I$ is non-empty. 
By condition~(4) the induced quiver morphism $(\rho_i)_i\colon Q\to \prod_{i\in I} A_i$ into the product is well-defined, and by condition~(5) it is again a representation of $(Q,R)$.
Let 
\[
\rho_{(Q,R)} \colon Q \longrightarrow \mathrm{U}(Q,R)
\]
be the restriction of this representation to the smallest sub-$\Cstar$-category of $\prod_{i\in I} A_i$ still containing the image of~$(\rho_i)_i$. 
By condition~(2), $\rho_{(Q,R)}$ is a representation of~$(Q,R)$, which moreover is dense by construction.
Let us verify that $\rho_{(Q,R)}$ satisfies the universal property of the theorem.

Let $\pi_i\colon \mathrm{U}(Q,R)\to A_i$ denote the restrictions to $\mathrm{U}(Q,R)$ of the canonical projections coming with the product. 
Note that each $\pi_i$ is a $*$-functor satisfying $\pi_i \cdot \rho_{(Q,R)} = \rho_i$.
If $\rho\colon Q \to A$ is an arbitrary representation of $(Q,R)$, it factors as a dense representation $\rho'\colon Q\to A'$ 
(by taking $A'$ to be the sub-$\Cstar$-category of $A$ generated by $\rho(Q)$ and because of~(2)), followed by the faithful inclusion $\iota\colon A' \to A$. 
By the definition of~$I$, there exist an $i_0\in I$ and an isomorphism of $\Cstar$-categories $\varphi\colon A_{i_0} \to A'$ such that $\varphi \cdot \rho_{i_0} = \rho'$.
We set 
$\tilde \rho := \iota \cdot \varphi \cdot \pi_{i_0} \colon \mathrm{U}(Q,R)\to A$. 
Then
\[
\tilde  \rho \, \rho_{(Q,R)} 
=  \iota \, \varphi \, \pi_{i_0} \, \rho_{(Q,R)}
= \iota \, \varphi \, \rho_{i_0}
= \iota \, \rho'
= \rho,
\]
and any $*$-functor $\mathrm{U}(Q,R)\to A$ satisfying this equality must agree with $\tilde \rho$ on~$Q$, and therefore is equal to $\tilde \rho$ by the density of~$\rho_{(Q,R)}$.
\end{proof}

\begin{example}[Algebraic relations]
 \label{ex:algebraic-relations}
Let $Q$ be any small quiver and let $R$ be any set of \emph{algebraic relations} on the arrows of~$Q$. That is, $R$ consists of equations between elements of the free $*$-category $\mathrm F(Q)$. Then $(Q,R)$ is admissible and the universal $\Cstar$-category $\mathrm{U}(Q,R)$ exists, as soon as condition (4) holds. Indeed (1), (2) and (3) are obvious, and (5) is satisfied because the $*$-algebraic operations in a product of $\Cstar$-categories are defined coordinatewise.
In particular, if $B$ is any $*$-category we can consider the pair $(Q(B),R(B))$, consisting of the underlying quiver $Q(B)$ and the set $R(B)$ of all algebraic relations of~$B$. Then the \emph{enveloping $\Cstar$-category} 
\[
\rho_B:=\rho_{(Q(B),R(B))} 
\colon B \longrightarrow \mathrm{U}(Q(B),R(B))=:\mathrm{U}(B)
\] 
of the $*$-category~$B$ exists, provided that
\begin{equation} \label{norm_condition}
\| b \|_{\infty}:=\sup \{ \| \rho(b) \| \mid \rho \colon B\to A \textrm{ is a}*\textrm{-functor to a }\Cstar\textrm{-category }  \} 
< \infty
\end{equation}
for all arrows $b \in B$. We see from the proof of Theorem~\ref{thm:univ} that the universal $*$-functor $\rho_B\colon B \to \mathrm{U}(B)$ (if it exists!) is surjective on objects and has dense image. Therefore the enveloping $\Cstar$-category $\mathrm{U}(B)$ is actually the completion of the $*$-category $B$ with respect to the (semi-)norm $\| \cdot \|_{\infty}$.

Note that, by the representation theorem for $\Cstar$-categories (Prop.\,\ref{prop:GNS}), in order to compute $\|\cdot \|_{\infty}$ it suffices to consider $*$-functors $B\to \hilb$. Also, the canonical $*$-functor $\rho_B$ is faithful if and only if there exists a faithful $*$-functor into any $\Cstar$-category, if and only if there exists a faithful $*$-functor $B\to \hilb$.
\end{example}

We now give a modest example of a universal $\Cstar$-category that is \emph{not} defined by algebraic relations. 
It will be needed to prove Proposition~\ref{prop:cof_gen}.

\begin{example} \label{ex:one_arrow}
Let $Q= \{a \colon 0\to 1\}$ be the quiver with one single arrow between two distinct objects, and consider the single relation $R= \{ \| a \| \leq 1\}$. It is immediately verified that $(Q,R)$ is admissible, and we shall denote the resulting universal $\Cstar$-category by~$\mathbf 1$.
 Note that $\mathbf 1$ has the property that, for any $\Cstar$-category~$A$, $*$-functors $F\colon \mathbf 1 \to A$ correspond bijectively to arrows $a\in A$ of norm at most equal to one. Since there exist operators of norm one, it follows that indeed $\|a\|=1$ in~$\mathbf 1$. For similar reasons, we see that $a^*a\neq1_0$ and $aa^*\neq1_1$.
\end{example}

Our next application is to show the cocompleteness of~$\Cstarcatun$.

\begin{lemma}\label{lemma:star-colims}  
The category of small $*$-categories and $*$-functors between them has all small colimits.
\end{lemma}

\begin{proof} It is enough to construct coproducts and coequalizers.  Coproducts $\coprod_iB_i$ are easy: just take the disjoint union $\coprod_iB_i$ on objects, with Hom spaces
$(\coprod_i B_i)(x,y):= B_j(x,y)$ 
for $x,y\in B_j$ and 
$(\coprod_iB_i)(x,y):=0$ 
for $x \in B_j, y \in B_k$ and $j\neq k$. 
Then $\coprod_i B_i$ inherits a unique composition and involution from those of the $B_i$'s, so that the canonical inclusions $B_j\to \coprod_iB_i$ are $*$-functors.
Coequalizers are slightly trickier. Let $F_1,F_2\colon B\to C$ be any two parallel $*$-functors between $*$-categories, and denote by $Q$ the forgetful functor to quivers. Clearly if we quotient the quiver $QC$ by the relations
$$F_1(x)\sim F_2(x)\quad (x\in\obj B),\quad F_1(f)\sim F_2(f)\quad(f\in B),$$
we obtain a quiver $D$ which is the coequalizer, as a quiver, of $QF_1$ and~$QF_2$. Then the coequalizer of $F_1$ and $F_2$ is the small $*$-category $\mathrm F(D)/I$, the quotient of the free $*$-category on $D$ by the $*$-closed ideal generated by all the relations of~$C$.
\end{proof}

\begin{prop} \label{prop:cocomplete}
The category $\Cstarcatun$ has all small colimits.
\end{prop}

\begin{proof} Let $X\colon \mathcal I\to \Cstarcatun$ be a small diagram. Let $\colim VX$ be the colimit of the diagram~$VX$, where $V$ is the forgetful functor to the category of small $*$-categories and $*$-functors. This exists by Lemma~\ref{lemma:star-colims}.
Now, by Example~\ref{ex:algebraic-relations} we only have to check~\eqref{norm_condition}, the boundedness of the supremum norm over all representations, in order for the enveloping $\Cstar$-category $\colim VX\to \mathrm{U}(\colim VX)$ to exist. It is then clear that this construction, together with the composite $*$-functors $X(i)\to \colim VX \to \mathrm{U}(\colim VX)$, enjoys the universal property of $\colim X$ in~$\Cstarcatun$.

Let $\rho\colon \colim VX\to A$ be any $*$-functor to some $\Cstar$-category. Let $f$ be an arrow of  $\colim VX$. By construction (see the proof of Lemma~\ref{lemma:star-colims}), $f$ is represented by some algebraic combination of arrows $f_k$ of the categories $X(i), i\in \obj \mathcal I$, where, as before, ``algebraic'' means that only finite linear combinations, compositions and adjoints are allowed. Call this combination $C(f_1,\ldots, f_n)$. Then the following holds, where $\tilde C$ indicates the obvious corresponding algebraic combination of the norms.
\begin{eqnarray*}
\|\rho f\|_A = \|\, \rho [C(f_1,\ldots,f_n)]\,\|_{A}&\leq& \tilde C(\, \|\rho[f_1]\|_A,\ldots,\|\rho[f_n]\|_A \,)\\
&\leq& \tilde C(\, \|f_1\|_{X(i_{1})},\ldots,\|f_n\|_{X(i_{n})}\,)=:c(f) 
\end{eqnarray*}
The first inequality is due to the triangle inequality and sub-mult\-iplic\-at\-iv\-ity of the norm and to the isometricity of the involution in~$A$. The second one holds because each $*$-functor $\smash{X(i)\to \colim VX\stackrel{\rho}{\to} A}$ starts and ends at $\Cstar$-categories and is therefore automatically norm-reducing. In particular, the bound $c(f)< \infty$ does not depend on the representation~$\rho$. Hence $\colim VX$ satisfies condition~\eqref{norm_condition}, and we are done.
\end{proof}

\begin{remark} 
While the inclusion $\Cstaralgun \hookrightarrow \Cstarcatun$, of (unital) $\Cstar$-algebras into $\Cstar$-cat\-eg\-ories is easily seen to preserve all limits, it does not preserve colimits for the obvious reasons, \emph{e.g.}: $\F \sqcup \F$ in $\Cstaralgun$ is the usual (unital) free product $\F* \F$, while $\F \sqcup \F$ in $\Cstarcatun$ has two objects and therefore is not even an algebra.
\end{remark}

%

\subsection{The maximal tensor product}

Given two $*$-categories $ A$ and $ B$, their \emph{algebraic tensor product} $ A\otimes_{\F} B$  is simply their tensor product as $\F$-categories
\[
\obj (A \otimes_{\F} B):= \obj A \times \obj B 
\quad,\quad
(A \otimes_{\F} B)((x,y),(x',y')) :=  A(x,x') \otimes_{\F} B(y,y')
\]
equipped with the involution  $(\sum_i z_ia_i\otimes b_i)^*:=\sum_i \overline{z_i}a_i^*\otimes b_i^*$. As in the case of $\Cstar$-algebras, if $A$ and $B$ are $\Cstar$-categories there are in general different ways to make the algebraic tensor product into a $\Cstar$-category.  One possibility is to complete with  respect to the spatial norm: Proposition~\ref{prop:GNS} provides faithful representations $\rho\colon A\to\hilb$ and $\sigma\colon B\to \hilb$, which we may combine to form a representation $\rho \otimes \sigma\colon A\otimes_{\F} B\to\hilb$ by sending $a\otimes b\colon (x,y)\to (x',y')$ to $\rho(a)\otimes \sigma(b) \in \mathcal L(\rho(x)\otimes \sigma(y), \rho(x')\otimes \sigma(y'))$ and extending linearly.
We can thus define  a norm
\begin{equation*}
\|f\|_{\min}:= \|(\rho \otimes \sigma)(f)\|_{\hilb}
\end{equation*}
on the algebraic tensor product $A\otimes_{\F}B$, which turns out to be independent of the choices of $\rho$ and~$\sigma$. This $\|\cdot \|_{\min}$ is a $\Cstar$-norm, and the corresponding completion $A \otimes_{\min} B$ of $A \otimes_{\F} B$ is called the \emph{minimal tensor product} of $A$ and~$B$ (see~\cite{mitchener:Cstar_cats}).

On the other hand, it is also possible to generalize the maximal tensor product of $\Cstar$-algebras:

\begin{prop}\label{prop:max} 
Let $A,B$ be two small $\Cstar$-categories.
The supremum norm \eqref{norm_condition} on the $*$-category $A \otimes_{\F} B$ is bounded, therefore the universal enveloping $\Cstar$-category of $A \otimes_{\F} B$ exists. 
Denote it by $A \otimes_{\max} B$ and call it the \emph{maximal tensor product of $A$ and~$B$}.
The canonical $*$-functor from $A \otimes_{\F} B$ into it is faithful, and the construction specializes to the usual maximal tensor product of \textup(unital\textup) $\Cstar$-algebras.

\end{prop}

\begin{proof} Let us check that $\|f\|_{\infty}$ is finite for each arrow $f \in A\otimes_{\F}B$. 
For every object $x \in\obj  A$, we can define a $*$-functor $J_x\colon B\to A\otimes_{\F} B$  by sending $b\colon y\to y'$ to $1_x\otimes b\colon (x,y)\to(x,y')$. In the same way we can define a $*$-functor $I_y\colon A\to A\otimes_{\F} B$ for every choice of $y \in\obj B$. 
Let $F$ be any representation of $ A\otimes_{\F} B$. 
Then the compositions $F I_y$ and $F J_x$ are $*$-functors between $\Cstar$-categories, and therefore they are norm-decreasing.
Using this fact and the commutative triangles 
\[
\xymatrix{
(x,y) \ar[dr]_{a_i \otimes b_i} \ar[r]^-{a_i \otimes 1_y} & 
 (x',y) \ar[d]^{1_{x'} \otimes b_i} \\
 &
 (x',y')
}
\]
in $A\otimes_{\F} B$, we obtain for every morphism  
$f=\sum_i z_ia_i\otimes b_i \in (A\otimes_{\F} B) ((x,y),(x',y'))$
\begin{eqnarray*}
\left\|F \left(\sum_i z_ia_i\otimes b_i \right) \right\|  
&\leq&
  \sum_i |z_i| \| F(a_i \otimes  b_i) \| \\
&= &
 \sum_i |z_i| \| F(1_{x'} \otimes b_i \, \cdot \,  a_i \otimes 1_y ) \| \\
&= &
  \sum_i |z_i| \| F(1_{x'}\otimes b_i) \cdot  F(a_i\otimes 1_{y})  \|    \\
&\leq&
  \sum_i |z_i|  \|FJ_{x'}(b_i)\|  \|FI_y (a_i)\|  \\
& \leq&  \sum_i |z_i| \|b_i\| \|a_i\|  , 
\end{eqnarray*} 
which is independent of~$F$. Therefore $\|f\|_{\infty} < \infty$. Thus the universal enveloping $\Cstar$-algebra for $A\otimes_{\F} B$ exists, as in Example~\ref{ex:algebraic-relations}. Moreover, since there exists a faithful representation (\emph{e.g.}, $A \otimes_{\F}B \hookrightarrow A \otimes_{\min} B$), the canonical map $A \otimes_{\F}B \to A \otimes_{\max} B$ is faithful. 
The last claim of the proposition is clear because, if $A$ and $B$ are $\Cstar$-algebras, then the $A\otimes_{\max} B$ just defined is a $\Cstar$-algebra satisfying the universal property of the maximal tensor product of the $\Cstar$-algebras $A$ and~$B$.
\end{proof}

\begin{remark} \label{rem:nuclear}
As with $\Cstar$-algebras, we may call a $\Cstar$-category \emph{nuclear} if the comparison $*$-functor $A\otimes_{\max} B \to A \otimes_{\min} B$ is an isomorphism for all $\Cstar$-categories~$B$. Presumably, we should expect nuclearity  to play a fundamental role in the general theory of $\Cstar$-categories, similar to the role of nuclearity in the general theory of $\Cstar$-algebras -- but we have not explored this line of thought yet.
\end{remark}

\begin{lemma} \label{lemma:adjunction_Cstar}
Let $C$ be any $\Cstar$-category, and let $A,B$ be small $\Cstar$-categories. Then the usual exponential law for \textup(ordinary\textup) categories induces an isomorphism
\[
\Cstar(A\otimes_{\max} B, C) \simeq \Cstar( A, \Cstar (B,C))
\]
of $\Cstar$-categories.
\end{lemma}

\begin{proof}
It is an exercise (and a basic results of enriched category theory) to verify that the exponential law for categories upgrades to $\F$-categories, in the form of an $\F$-linear isomorphism
\begin{equation} \label{basic_F_lin}
\Phi\colon \inthom_{\F}(A \otimes_{\F} B , C) \simeq \inthom_{\F} (A, \inthom_{\F} (B,C))
\end{equation}
 ($\inthom_{\F}$ denotes the $\F$-category of $\F$-linear functors and natural transformations bet\-ween them).
Under~$\Phi$, an $\F$-functor $F\colon A\otimes_{\F} B \to C $ is sent to the $\F$-functor $\Phi F\colon A \to \inthom_{\F}(B,C)$ defined as follows. 
On objects, $\Phi F$ sends $x\in \obj A$ to $F( 1_x \otimes - )\colon B\to C$. 
An arrow $a\colon x\to x'$ is sent to the natural transformation 
$\Phi F(a)= F( a \otimes - ){:}\break F( 1_x \otimes - ) \to F( 1_{x'} \otimes - )$. 
The inverse $\Psi$ of $\Phi$ takes a functor $G\colon A \to \inthom_{\F}(B,C)$ and assigns to it the unique $\F$-functor $\Psi G\colon A\otimes_{\F} B\to C$ sending  $a \otimes b \colon (x,y) \to (x',y') $
to 
$\Psi G (a \otimes b):=  G(a)_{y'} \cdot G(x)(b) = G(x')(b) \cdot G(a)_y$.

Note that, if $D,E$ are two $*$-categories, the full subcategory of $\inthom_{\F}(D,E)$ of $*$-functors (let us denote it by $\inthom_*(D,E)$) is also a $*$-category for the pointwise involution $\alpha^*(f):= \alpha(f)^*$ on natural transformations. 
It is straightforward to verify that, if $A,B,C$ are $*$-categories, \eqref{basic_F_lin} restricts to an isomorphism of $*$-categories
\begin{equation} \label{*-cats}
\Phi: \inthom_*(A \otimes_{\F} B , C) \simeq \inthom_* (A, \inthom_* (B,C)) : \Psi.
\end{equation}
If $D,E$ are moreover $\Cstar$-categories, then $\Cstar(D,E)$ is the (possibly non-full) subcategory of $\inthom_*(D,E)$ on the same objects, where morphisms are bounded transformations.
Now let $A,B,C$ be $\Cstar$-categories. 

If $F\colon A\otimes_{\max} B\to C$ is a $*$-functor, we may restrict it to $A\otimes_{\F} B$ and produce a $*$-functor 
$\Phi F \colon A\to  \inthom_*(B,C)$. 
We claim that $\Phi F$ is actually a functor $A\to \Cstar(B,C)$. Indeed, each arrow $a\in A(x,x')$ is sent to the natural transformation $\Phi F(a)$ with components $\Phi F(a)_y= F(a \otimes 1_y)$ ($y\in \obj B$). 
Thus we see that
$\| \Phi F (a) \|_{\infty}$ $=$ $\sup_y \| F(a \otimes 1_y) \|_C = \sup_y \| F I_y (a) \| \leq \|a\|_A $, 
again using the observation that each $*$-functor $FI_y\colon A \to A\otimes_{\max} B \to C$ is norm-decreasing, with the same notation $I_y$ as in the proof of Prop.\,\ref{prop:max}. This proves the claim.

On the other hand, every $*$-functor $A \to \Cstar(B,C)$ gives rise to
a $*$-functor $\Psi G{:}\break A \otimes_{\F} B\to C$ and, since $C$ is a $\Cstar$-category, this extends uniquely to a $*$-functor $A \otimes_{\max} B \to C$, which we unflinchingly still denote by $\Psi G$.
Thus $\Phi$ and $\Psi$ define mutually inverse bijections
\[
\Phi\colon \obj \inthom_*(A \otimes_{\max} B , C) \simeq \obj \inthom_* (A, \Cstar (B,C)) : \Psi.
\]
Since this is all we need for the next theorem, we leave to the reader the straightforward verification that $\Phi$ and $\Psi$ extend to an isomorphism of $\Cstar$-categories as claimed (actually, this will also be a formal consequence of Theorem~\ref{thm:closed_monoidal_Cstarcatun}).
\end{proof}

\begin{thm} \label{thm:closed_monoidal_Cstarcatun}
The maximal tensor product $\otimes_{\max}$ endows $\Cstarcatun$ with the structure of a closed symmetric monoidal category, 
with unit object $\F$ and internal Hom's the $\Cstar$-categories $\Cstar(A,B)$. 
\end{thm}

\begin{proof}
Using the universal property of the maximal tensor product, it is straightforward routine to verify that it defines a functor $-\otimes_{\max}-\colon \Cstarcatun \times \Cstarcatun \to \Cstarcatun$, and that the structural associativity, left and right identity, and symmetry isomorphisms of the symmetric monoidal structure $\otimes_{\F}$ of small $\F$-categories induce similar isomorphisms for~$\otimes_{\max}$, which again satisfy the axioms for a symmetric monoidal category with unit object~$\F$.
The internal Hom construction induces via composition of $*$-functors a functor $\Cstar(-,-)\colon \Cstarcatun^{\op} \times \Cstarcatun \to \Cstarcatun$, and 
the bijection 
\[
\Cstarcatun (A \otimes_{\max} B, C) \simeq \Cstarcatun (A, \Cstar(B,C))
\]
of Lemma~\ref{lemma:adjunction_Cstar}, which is readily seen to be natural in $A,B,C\in \Cstarcatun$, shows that the monoidal structure is closed with internal Hom given by $\Cstar(-,-)$.
\end{proof}

\subsection{The maximal groupoid $\Cstar$-category}
\label{subsec:max_gpd}
In order to define the simplicial structure on $\Cstarcatun$ we shall need the maximal (or ``full'') groupoid $\Cstar$-category $\Cstarmax \mathcal G$ associated to a small discrete groupoid. We recall its construction from~\cite{mitchener:Cstar_cats}*{Def.\,5.10}, but here we shall rather emphasize the universal property it enjoys. With this perspective, we can show that it transforms products of groupoids into maximal tensor products: 
$\Cstarmax(\mathcal G_1 \times \mathcal G_2)\simeq \Cstarmax \mathcal G_1 \otimes_{\max} \Cstarmax \mathcal G_2$.

\begin{definition} \label{defi:max_gpd}
The \emph{\textup(maximal\textup) groupoid $\Cstar$-category} associated to a groupoid\footnote{All groupoids in this article are discrete.}~$\mathcal G$ is  the universal functor $\rho_{\mathcal G}\colon \mathcal G \to \Cstarmax \mathcal G$ mapping all arrows of $\mathcal G$ to unitary elements of a  $\Cstar$-category.
Clearly this defines a functor $\Cstarmax\colon \Gpd\to \Cstarcatun$ on the category of small groupoids and functors between them, which sends equivalences of groupoids to unitary equivalences of $\Cstar$-categories.
\end{definition}

To see that the universal notion just described actually exists, we can realize it as the enveloping $\Cstar$-category 
(see Example~\ref{ex:algebraic-relations})
of the $*$-category $\F\mathcal G$ with
\[
(\F \mathcal G)(x,x') := \F  \mathcal G(x,x') \quad (\textrm{free } \F\textrm{-module})
\quad,\quad
 \left(\sum_i z_i g_i \right)^*:= \sum_i \overline{z_i} \, g_i^{-1}
\]
for  $x,x'\in \obj \mathcal G$ and $\sum_i z_i g_i \in \F\mathcal G(x,x')$.
(This $\F\mathcal G$ is the \emph{groupoid category} of~\cite{mitchener:Cstar_cats}*{Def.\,5.4}.)
Thus it suffices to verify that $\|f \|_{\infty}= \sup_{\rho} \|\rho (f) \|$, with supremum taken over all $*$-functors $\rho \colon \F\mathcal G\to \hilb$, is finite for every arrow $f\in \F\mathcal G$. Indeed we have
$\|\rho (\sum_i z_i g_i) \| 
\leq \sum_i |z_i| \| \rho(g_i) \| 
\leq \sum_i |z_i| < \infty$,
because unitaries have norm one or zero (the latter happens if the unitary's domain and codomain have zero endomorphism algebras).

\begin{remark}
The completion of $\F\mathcal G$ with respect to another, more concretely defined norm $\|\cdot \|_r$, produces the \emph{reduced} groupoid $\Cstar$-category~$\mathrm C^*_r \mathcal G$, which generalizes the reduced $\Cstar$-algebra of a group
(see~\cite{mitchener:Cstar_cats}*{Def.\,5.8}). Since the canonical $*$-functor $\F\mathcal G\to \mathrm C^*_r \mathcal G$ is by construction a faithful representation, we conclude that the canonical $*$-functor $\F\mathcal G \to \Cstarmax \mathcal G $ is also faithful. 
\end{remark}

For any $\Cstar$-category~$A$, denote by $\unitaries A$ its subcategory of unitary elements.

It follows immediately from Definition~\ref{defi:max_gpd} that composition with the faithful functor 
$\rho_{\mathcal G}\colon \mathcal G\to \unitaries \Cstarmax \mathcal G$ induces a natural bijection
\begin{equation}  \label{general_bij}
\Cstarcatun(\Cstarmax \mathcal G, A)\simeq \Gpd (\mathcal G, \unitaries A)
\end{equation}
for any small groupoid $\mathcal G$ and any, possibly large, $\Cstar$-category~$A$. In particular:

\begin{prop} \label{prop:Cstar_adjoint}
The functor $\unitaries \colon \Cstarcatun \to \Gpd$, assigning to a $\Cstar$-category~$A$ its groupoid $\unitaries A$ of unitary elements, has a left adjoint $\Cstarmax \colon \Gpd \to \Cstarcatun$. 
\qed
\end{prop}

The next lemma exhibits a 2-categorical upgrade of the universal property of the groupoid $\Cstar$-category. 
Here we denote by $\underline{\mathrm{Hom}}(C,D)$ the category of functors between two categories $C,D$.

\begin{lemma} \label{lemma:UP_max}
For every $\Cstar$-category $A$ and small groupoid $\mathcal G$, composition with $\rho_{\mathcal G}{:}\break\mathcal G\to \unitaries  \Cstarmax \mathcal G$ induces an isomorphism of groupoids
\[
\unitaries \Cstar(\Cstarmax \mathcal G, A )
\simeq \underline{\mathrm{Hom}} (\mathcal G, \unitaries A).
\]
\end{lemma}

\begin{proof}
The bijection on objects is~\eqref{general_bij}.
Now let $F, F'\colon \Cstarmax \mathcal G \to A$ be any two $*$-functors, and let $\alpha\colon F\to F'$ be a unitary isomorphism. That is, for every $$x\in \obj(\Cstarmax \mathcal G)= \obj \mathcal G$$ we have a unitary arrow $\alpha_x\colon F(x)\to F'(x)$ in~$A$, and for every arrow $f\in\break (\Cstarmax \mathcal G)(x,x')$ the square
\[
\xymatrix{
F(x) \ar[r]^-{F(f)} \ar[d]_{\alpha_x} &
 F(x') \ar[d]^{\alpha_{x'}} \\
F'(x) \ar[r]^-{F'(f)} &
 F'(x')
}
\]
is commutative. In particular, it commutes for every $f\in \mathcal G(x,x')$, so $\alpha$ can also be seen as an (iso)morphism $F\rho \to F' \rho$ in the category 
$ \underline{\mathrm{Hom}} (\mathcal G, \unitaries A)$. 
Conversely, assume that we have a collection $(\alpha_x\colon Fx\to F'x)_{x\in \obj \mathcal G}$ of unitaries in $A$ rendering the above squares commutative for all $f\in \mathcal G(x,x')$. Then by the linearity of composition in~$A$, the squares commute for all $f\in \F\mathcal G(x,x')$, and by continuity of composition they commute for all $f$ in the completion $\Cstarmax\mathcal G(x,x')$, for all~$x,x'$. Thus $\alpha=(\alpha_x)_{x}$ is a unitary isomorphism $F\to F'$ in $\Cstar(\Cstarmax \mathcal G,A)$.
 \end{proof}

\begin{lemma} \label{lemma:max_products}
Let $\mathcal G_1$ and $ \mathcal G_2$ be two small groupoids and $A$ be any $\Cstar$-category.
There is an isomorphism of groupoids
\[
 \unitaries \Cstar \big( (\Cstarmax \mathcal G_1) \otimes_{\max} (\Cstarmax\mathcal G_2) , A \big)
\simeq
\unitaries \Cstar \big( \Cstarmax(\mathcal G_1 \times \mathcal G_2), A \big) 
\]
obtained by restricting and extending functors along the canonical maps 
$\mathcal G_1 \times \mathcal G_2 \to \Cstarmax(\mathcal G_1 \times \mathcal G_2)$ and
$\mathcal G_1 \times \mathcal G_2 \to \Cstarmax \mathcal G_1 \otimes_{\max} \Cstarmax \mathcal G_2 $.
\end{lemma}

\begin{proof}
We get the following  isomorphisms of categories
\begin{eqnarray*}
\unitaries \Cstar(\Cstarmax(\mathcal G_1 \times \mathcal G_2), A )
& \simeq & \underline{\mathrm{Hom}} (\mathcal G_1 \times \mathcal G_2, \unitaries A) \\
& \simeq & \underline{\mathrm{Hom}} (\mathcal G_1 , \underline{\mathrm{Hom}} (\mathcal G_2, \unitaries A) ) \\
&\simeq & \underline{\mathrm{Hom}} ( \mathcal G_1, \unitaries \Cstar (\Cstarmax \mathcal G_2 , A) ) \\
&\simeq & \unitaries \Cstar ( \Cstarmax \mathcal G_1 , \Cstar (\Cstarmax \mathcal G_2 , A) ) \\
&\simeq & \unitaries \Cstar(\Cstarmax \mathcal G_1 \otimes_{\max}\Cstarmax \mathcal G_2, A )
\end{eqnarray*}
by successively applying: the universal property of 
$\Cstarmax(\mathcal G_1\times \mathcal G_2)$, 
as in Lemma~\ref{lemma:UP_max};
the standard closed monoidal structure of $\Gpd$;
the universal property of~$\Cstarmax\mathcal G_2$ (after which one applies $\underline{\mathrm{Hom}} (\mathcal G_1, -)$);
the universal property of $\Cstarmax\mathcal G_1$;
and finally the exponential law of Lemma~\ref{lemma:adjunction_Cstar}, to which one applies~$\unitaries$. 
By looking closely at the constructions of these isomorphisms, we see that the above composition is  the result of extending functors and restricting $*$-functors along the canonical maps, as claimed. 
\end{proof}

Now consider the comparison map
\begin{equation} \label{can_comp}
\Cstarmax(\mathcal G_1 \times \mathcal G_2)
\to  \Cstarmax \mathcal G_1 \otimes_{\max} \Cstarmax\mathcal G_2
\end{equation}
induced by the universal property of $\Cstarmax(\mathcal G_1 \times \mathcal G_2)$. It is the unique $*$-functor extending 
$\mathcal G_1 \times \mathcal G_2 \to  \Cstarmax \mathcal G_1 \otimes_{\max} \Cstarmax\mathcal G_2$,
 $(g_1,g_2) \mapsto g_1\otimes g_2$,
and it fits in a commutative triangle
\[
\xymatrix{
&\mathcal G_1\times \mathcal G_2 \ar[dl] \ar[dr] &
\\
\Cstarmax(\mathcal G_1 \times \mathcal G_2 ) \ar[rr] &&
 \Cstarmax \mathcal G_1 \otimes_{\max} \Cstarmax \mathcal G_2
}
\]
with the two canonical maps in Lemma~\ref{lemma:max_products}. From the lemma it follows in particular that $\Cstarcatun(- ,A)$ applied to~\eqref{can_comp} yields a bijection for every small $\Cstar$-category~$A$. 
We conclude by the Yoneda lemma (\cite{maclane}*{III.2, p.\,62}) that \eqref{can_comp} is an isomorphism in~$\Cstarcatun$. 
 Moreover:

\begin{thm} \label{thm:max_products}
The natural isomorphism \eqref{can_comp} and the canonical identification $\F\simeq \F1=\Cstarmax(1)$ make $\Cstarmax$ into a \textup(strong\textup) symmetric monoidal functor from the symmetric monoidal category of groupoids, $(\Gpd, \times, 1)$, to the symmetric monoidal category of $\Cstar$-categories, $(\Cstarcatun, \otimes_{\max}, \F)$.
\end{thm}

\begin{proof}
The commutativity of the coherence diagrams involving~\eqref{can_comp}, $\F\simeq \Cstarmax(1)$, and the structural isomorphisms of the two symmetric monoidal categories (see~\cite{maclane}*{VI}), follows immediately from the uniqueness of the arrows induced by the universal properties of $\Cstarmax$ and~$\otimes_{\max}$.
\end{proof}

\section{The unitary model structure}

We shall now prove the main theorem, that there exists a well-endowed model structure on the category of small unital $\Cstar$-categories, whose weak equivalences are precisely the unitary equivalences (Def.\,\ref{defi:unitary_eq}). This section owes much to Charles Rezk's neat presentation~\cite{rezk:folk} of the canonical (or ``folk'') model structure on the category of small categories. Our reference for the theory of model categories will be~\cite{hovey:model}; for a pleasant gentle introduction we refer to the expository article~\cite{dwyer-spalinski}.

\begin{definition} \label{defi:cof_fib}
Let $F\colon A\to B$ be a $*$-functor between unital $\Cstar$-categories.
We call $F$  a \emph{cofibration} if the map on objects $\obj F\colon \obj A\to \obj B$ is injective. 
We call $F$ a \emph{fibration} if for every object $y\in \obj B$ and every unitary isomorphism $v\colon Fx\to y$ there exists a unitary $u\colon x\to x'$ in $A$ such that $Fu=v$ (and therefore $Fx'=y$). 
\end{definition}

\begin{thm}[The unitary model structure]
 \label{thm:first_model}
The class $\Cof$ of cofibrations and $\Fib$ of fibrations as in Definition~\ref{defi:cof_fib}, together with the class $\Weq$ of unitary equivalences as weak equivalences, define a model structure $(\Cstarcatun, \Weq, \Fib, \Cof)$, that we shall call the \emph{unitary model structure}.
\end{thm}

\begin{remark} \label{rem:fibrant_cofibrant}
Every $\Cstar$-category $A$ is both fibrant and cofibrant in the unitary model, \emph{i.e.}, the unique $*$-functor $A\to \mathbf 0$ to the final $\Cstar$-category $\mathbf 0$ (the one having one single object with zero endomorphism space) is a fibration, and the unique $*$-functor $\emptyset \to A$ from the initial (\emph{i.e.}, empty) $\Cstar$-category $\emptyset$ is a cofibration.

\end{remark}

\begin{definition} \label{defi:interval}
Define the \emph{interval  $\Cstar$-category} $\mathbf I$ to be the universal unital $\Cstar$-category with two objects, $0$ and~$1$, and a unitary element $u\colon 0\to 1$. (In other words, $\mathbf I=\Cstarmax(I)$, where $I$ is the groupoid with two objects $0,1$ and a single isomorphism between them). Thus $*$-functors $\mathbf I\to A$ correspond to unitaries $u\in A$.
\end{definition}

\begin{lemma} \label{lemma:lifting_fib}
A $*$-functor $F\colon A\to B$ is a fibration if and only if it has the right lifting property with respect to the $*$-functor $0\colon \F\to \mathbf I$ corresponding to~$0\in \obj \mathbf I$, \emph{i.e.}, if and only if for every commutative square in $\Cstarcatun$ of the form
\[
\xymatrix{
\F \ar[r] \ar[d]_0 & 
 A \ar[d]^F \\
  \mathbf I \ar[r] \ar@{..>}[ur] &
   B
}
\]
 there exists a dotted map as indicated, making the two triangles commute.
\end{lemma}

\begin{proof}
This is clear, because $*$-functors $\F\to A$ are in bijection with the objects of~$A$
 and $*$-functors $\mathbf I\to B$ with the unitary elements of~$B$.
\end{proof}

Although the next lemma is obvious (in view of Proposition~\ref{prop:uni_eq}), we spell out the argument again as this will later help us clarify some constructions. 

\begin{lemma} \label{lemma:un_eq}
Let $F\colon A\to B$ be a $*$-functor between $\Cstar$-categories. The following are equivalent:
\begin{itemize}
\item[(i)] $F$ is a unitary equivalence.
\item[(ii)] $F$ is fully faithful \textup(as a functor of the underlying categories of $A$ and~$B$\textup) and also \emph{unitarily essentially surjective}, i.e., for every $y\in B$ there is a unitary arrow $y \simeq Fx$ in $B$ for some $x\in \obj A$.
\item[(iii)] $F$ is an \textup(ordinary\textup) equivalence of the underlying categories of $A$ and~$B$.
\item[(iv)] $F$ is fully faithful and essentially surjective, in the usual sense.
\end{itemize}
\end{lemma}

\begin{proof}
The implications (i) $\Rightarrow$~(ii), (ii) $\Rightarrow$~(iv), and (iii) $\Rightarrow$~(iv) are seen immediately, and (iv) $\Rightarrow$~(ii) follows from Proposition~\ref{prop:uni_eq}. 
The implication (iv) $\Rightarrow$~(iii) is an exercise application of the axiom of choice, and (ii) $\Rightarrow$~(i) is quite similar, as we now verify in detail. Assume that $F\colon A\to B$ is fully faithful and unitarily essentially surjective. 
By the latter condition, we are free to choose for each $y\in \obj B$ an object $Gy \in \obj A$ and a unitary isomorphism $v_y\colon FGy\to y$.
Let $b\in B(y,y')$. Since $F$ is bijective on Hom spaces, we may set 
\[
Gb:= F^{-1}(v_{y'}^* b v_y) \in A(Gy,Gy').
\]
One checks immediately that the assignments $x\mapsto Gx$ and $b\mapsto Gb$ define a functor $G\colon B\to A$. 
Moreover, $G$ is a $*$-functor, as one sees by applying $F^{-1}$ to the equality
\[
F((Gb)^*) 
= (v_{y'}^* b v_y)^*
= v_{y}^* b^* v_{y'}
= FG(b^*),
\]
and the unitaries $(v_y)_{y\in \obj B}$ define an isomorphism $v\colon FG\to \id_{B}$. 
Finally, note that for every $x\in \obj A$ we have chosen an object $GFx\in\obj A$ and a unitary element $v_{Fx}\in B(FGFx,Fx)$.  
It is straightforward to verify that the elements $u_x:=F^{-1}(v_{Fx}) \in A(GFx,x)$ are unitary and form an isomorphism $u\colon GF\to \id_A$.
Therefore $F$ is a unitary equivalence with quasi-inverse~$G$.
\end{proof}

\begin{remark}
\label{remark:un_eq}
When proving the implication (ii)$\Rightarrow$(i) of Lemma~\ref{lemma:un_eq}, note that if~$F$ happens to be injective on objects, for every $y=F(x)\in \obj B$ lying in the image of~$F$ we may well choose $G(y):=x$, and similarly we may choose the unitary $v_y\colon FGy\to y$ to be the identity of~$y$.
The resulting quasi-inverse $G\colon B\to A$ has then the additional property that $GF=\id_A$.
\end{remark}

\begin{cor} \label{cor:char_tr_fib}
A $*$-functor $F\colon A\to B$ is a trivial fibration \textup(\emph{i.e.}, $F \in \Weq \cap \Fib$\textup) if and only if it is fully faithful and surjective on objects. 
\end{cor}

\begin{proof}
Assume that $F$ is a trivial fibration. 
Being a unitary equivalence, it is fully faithful and unitarily essentially surjective  (Lemma~\ref{lemma:un_eq}). 
Hence, for every object $y\in \obj B$ there are an $x\in \obj A$ and a unitary $v\colon y\simeq Fx$. 
Since $F$ is also a fibration, we may lift $v$ to a unitary $u\colon x'\simeq x$ in $A$ with $Fu=v$.
In particular $Fx'=y$, showing that $F$ is surjective on objects.

Conversely, assume that $F$ is fully faithful and surjective on objects. Then $F$ is unitarily essentially surjective (because identity maps are unitaries), and therefore it is a unitary equivalence by Lemma~\ref{lemma:un_eq}. Also, given a unitary $v\colon Fx\to y$ there is an $x'$ with $Fx'=y$, and~$F$ -- being a $*$-functor --  induces a bijection $F\colon A(x,x')\simeq B(Fx,y)$ of unitary elements.
Hence $F$ is a fibration too.
\end{proof}

We now prove the five axioms MC1-MC5 of a model categories.

 \textbf{MC1} (Limits and colimits axiom).
We have already established in Lemma~\ref{lemma:complete} that $\Cstarcatun$ is complete,
and in Proposition~\ref{prop:cocomplete} that it is cocomplete, so the first axiom holds.

 \textbf{MC2} (2-out-of-3 axiom). We must verify that unitary equivalences have the 2-out-of-3 property, \emph{i.e.}, that whenever two $*$-functors $F,G$ are composable and two out of $\{F,G, FG\}$ are in $\Weq$, then so is the third. This follows by combining the 2-out-of-3 property of the usual equivalences of categories with that of unitary arrows of $\Cstar$-categories. We leave the easy exercise to the reader.

\begin{lemma}[\textbf{MC3}, retract axiom]
\label{lemma:MC3}
Each of the classes $\Weq$, $\Cof$ and $\Fib$ is closed under taking retracts.
\end{lemma}

\begin{proof}
Fibrations are closed under retracts because they are characterized by a lifting property (Lemma~\ref{lemma:lifting_fib}). 
Cofibrations, because they are characterized by the injectivity of their underlying object-function, and injective maps are closed under retracts. There remain weak equivalences, which we check directly.
Consider a commutative diagram of $*$-functors
\[
\xymatrix{
A' \ar[d]_{F'} \ar[r]^-I &
 A \ar[d]_{F} \ar[r]^-P &
  A' \ar[d]^-{F'} \\
B' \ar[r]^-J &
 B \ar[r]^-Q &
  B'  
}
\]
with $PI=\id_{A'}$, $QJ=\id_{B'}$ and with $F$ a unitary equivalence. 
Choose a quasi-inverse $G \colon B\to A $ for $F$ and natural unitaries $u\colon GF \simeq \id_A$ and $v \colon FG\simeq \id_B$. 
We claim that $G':= PGJ\colon B'\to A'$ is a quasi-inverse for~$F'$. 
Indeed, $u$ and $v$ define two natural morphisms 
\[
Pu_{I} \colon G'F'=PGJF'=PGFI \longrightarrow PI=\id_{A'}
\] 
and
\[
Qv_J \colon F'G'=F'PGJ= QFGJ \longrightarrow QJ= \id_{B'}
,
\]
which moreover are unitary, because $*$-functors preserve unitaries. 
\end{proof}

\begin{lemma}[\textbf{MC4}, lifting axiom] 
\label{lemma:MC4}
 Consider a commutative square in $\Cstarcatun$
 \[
\xymatrix{
A \ar[r]^-U \ar[d]_F & 
 C \ar[d]^G \\
  B \ar[r]_-V \ar@{..>}[ur]^-L &
   D
}
\]
and assume either one of (i) or (ii): 
 \begin{itemize}
 \item[(i)]
 $F\in  \Cof \cap \Weq$ and $G\in \Fib$, or 
 \item[(ii)]
 $F\in \Cof$ and $G\in \Fib\cap \Weq$.
 \end{itemize}
 Then there exists a lifting $L\colon B\to C$ making the two triangles commute.
\end{lemma}

\begin{proof}
(i)
Let $F\in \Weq\cap \Cof$.
By Lemma~\ref{lemma:un_eq} and Remark~\ref{remark:un_eq}, there exists a quasi-inverse $F'\colon B\to A$ such that $F'F= \id_A$, and such that the other unitary $v\colon FF'\to \id_B$ is the identity for every object in the image of~$\obj F$:
\begin{equation} \label{identity_zero}
 v_{Fz}= 1_{Fz} \textrm{ for all } z\in \obj A.
 \end{equation} 
Let us first define the lifting $L$ on objects.  
For every $x\in \obj B$, let $y_x:=UF'x \in \obj C$. Note that $y_x$ is such that $Gy_x=VFF'x$, by the commutativity of the square. 
Since $G$ is a fibration and $Vv_x \colon Gy_x=VFF'x\to Vx$ is a unitary in~$D$, we can choose an object $Lx\in \obj C$ and a unitary $w_x\colon y_x \to Lx$ in~$C$ such that
\begin{equation} \label{identity_one}
GLx = Vx
\quad \textrm{ and } \quad
Gw_x= Vv_x.
\end{equation}
Moreover, for $x=Fz$ in the image of $\obj F$, we may certainly (and will) choose 
\begin{equation} \label{identity_two}
LFz = Uz  
\quad \textrm{ and } \quad
w_{Fz} = 1_{Uz}.
\end{equation}
We should now (and will) define $L$ on morphisms $b\colon x\to x'$ by the formula: 
\[
L b  :=  w_{x'}  \cdot  UF'b  \cdot  w_{x}^* ,
\]
so that the following square commutes.
\[
\xymatrix{
UF'x \ar[r]^-{w_x} \ar[d]_{UF'b} & 
 Lx \ar[d]^{Lb} \\
  UF'x' \ar[r]^-{w_{x'}}  &
   Lx'
}
\]
Using \eqref{identity_zero}, \eqref{identity_one}, \eqref{identity_two} and $F'F=\id_A$, it is now immediate to verify that $L\colon B \to C$ is a $*$-functor such that  $GL=V$ and $LF=U$. 
\smallbreak
(ii)
This time, we can choose the map $\obj L$ by using the injectivity of $\obj F$ and the surjectivity of~$\obj G$ (Corollary~\ref{cor:char_tr_fib}).
We take care to set $Lx:= Uz$ whenever $x=Fz$ is in the image of~$F$.
Since $G$ is also fully faithful, it provides isomorphisms $G \colon C(Lx, Lx') \simeq D(GLx,GLx')= D(Vx,Vx')$ for all $x,x' \in \obj C$.
Composing their inverses with $V$ defines the unique $*$-functor $L\colon C\to D$ with the chosen object-map and such that $GL=V$ and $LF=U$.
\end{proof}

\begin{lemma}[\textbf{MC5}, factorization axiom]
For every $*$-functor $F\colon A\to B$, there exist two factorizations 
\begin{itemize}
\item[(i)] $F= PI$ with $P\in \Fib $ and $I\in \Cof \cap \Weq$, and
\item[(ii)] $F= QJ$ with $Q\in \Fib \cap \Weq$ and $J\in \Cof $,
\end{itemize}
which are functorial in~$F$.
\end{lemma}

\begin{proof}
Our functorial factorizations will have the classical and familiar form: 
\[
\xymatrix{
A \ar[rr]^-F \ar[dr]_{I}^{\sim} && B
 \\
&
 \Cstar(\mathbf I, B) \times_B A \ar[ur]_{P} &
}
\quad \quad \quad
\xymatrix{
A \ar[rr]^-F \ar[dr]_{J} && B
 \\
&
 (A \otimes_{\max} \mathbf I) \sqcup_A B \ar[ur]^{\sim}_{Q}&
}
\]
The two midway objects are defined by the pullback, resp.\ pushout, squares
\[
\xymatrix{
\Cstar(\mathbf I, B) \times_B A \ar[d]_{\tilde F} \ar[r]^-{\tilde \ev_0} &
 A \ar[d]^-F
  \\  
\Cstar(\mathbf I, B) \ar[r]^-{\ev_0} & B
}
\quad\quad\quad
\xymatrix{
{\phantom{M}} A \otimes_{\max} \F = A \ar[r]^-F \ar[d]_{\id_A\otimes_{\max} 0} & 
B \ar[d]
\\
A \otimes_{\max} \mathbf I \ar[r] &
 (A \otimes_{\max} \mathbf I) \sqcup_A B
}
\]
(for~$\mathbf I$, see Def.~\ref{defi:interval}).
Here, the evaluation $*$-functor $\ev_0\colon \Cstar(\mathbf I,B)\to B$ sends a $*$-functor $F\colon \mathbf I\to B$ to $F0 \in \obj B$ and a morphism $\alpha\colon F\to F'$ to $\alpha_0\in B(F0,F'0)$.
\smallbreak
(i) Then we define $I=(C_B\circ F , \id_A)$ to be the $*$-functor with components $\id_A$ and 
$\smash{C_B F\colon A \stackrel{F}{\to} B \stackrel{C_B}{\to} \Cstar(\mathbf I, B)}$, where the `constant paths' $C_B$ is the $*$-functor that sends an $x\in \obj B$ to the unique $*$-functor $\mathbf I\to B$ which assigns $1_x$ to the generating unitary $0\to 1$, and sends an element $b\in B(x,x')$ to the morphism $C_Bx \to C_Bx'$ with component $b$ both at $0$ and~$1$. 
The $*$-functor $P$ is obtained by projecting to the first component and then evaluating at $1\in \obj \mathbf I$, that is, $P= \ev_1 \circ \tilde F$.

Clearly, by construction, $PI=F$ and $I$ is a cofibration. To see that $P$ is a fibration and $I$ a unitary equivalence, consider the following alternative description of $\Cstar(\mathbf I, B) \times_B A$: let $\tilde A$ be the $\Cstar$-category with 
\[
\obj \tilde A:= \{ (x,u,y) \mid x\in \obj A, y \in \obj B, u\colon Fx\simeq y \textrm{ unitary in } B  \}
\]
and with morphism spaces 
\[
\tilde A \big( (x,u,y), (x',u',y')\big) := A(x,x').
\]
This is a $\Cstar$-category with the evident operations, and it comes equipped with the two projections $\tilde \ev_0\colon\tilde A\to A$ and  $\tilde F\colon \tilde A\to \Cstar(\mathbf I,B)$ 
(the latter sending $(x,u,y)$ to the unique $*$-functor $\tilde F(x,u,y)\colon \mathbf I\to B$ which assigns the unitary $u\colon Fx \simeq y$ to the generator $0\to 1$, and sending $a\colon (x,u,y)\to (x',u',y')$ to the morphism $\tilde Fa\colon  \tilde F(x,u,y)\to \tilde F(x',u',y')$ with components $(\tilde Fa)_0:= Fa\colon  Fx\to Fx'$ at $0$ and $(\tilde Fa)_1:=u'\cdot Fa \cdot u^*\colon  y\to y'$ at~$1$). 
It is now straightforward to check that $\tilde A$, together with its two projections, satisfies the universal property of the pullback $\Cstar(\mathbf I,B)\times_B A$. 
With this picture, it is easy to see that $I\colon A\to \tilde A$ is given by $x\mapsto (x,1_x,Fx)$. Thus it is clearly fully-faithful and also unitarily essentially surjective, since for every $(x,u,y)\in \obj \tilde A$ the identity $1_x$ defines a unitary isomorphism $1_x\colon (x,u,y)\simeq (x,1_{Fx},Fx)=Ix$; therefore~$I$ is a unitary equivalence by Lemma~\ref{lemma:un_eq} (a quasi-inverse is given by $\tilde \ev_0\colon \tilde A\to A$). 

Similarly, the unitary $1_x \colon  (x,1_{Fx},Fx) \simeq (x,u',y')$ in $\tilde A$ provides a lift along
 \[
 P= \ev_1  \tilde F 
  \colon  
 \big( a\colon  (x,u,y) \to (x',u',y') \big)
 \mapsto 
 \big(  u'F(a)u^* \colon  y \to y' \big)
 \]
  of any unitary in $B$ of the form $u'\colon Fx\to y'$, showing that $P$ is a fibration.
\smallbreak
(ii) 
Similarly to part~(i), also for the second factorization it is convenient to consider a simpler model $\tilde B$ for the pushouts $(A\otimes_{\max} \mathbf I)\sqcup_A B$, defined by
\[
\obj \tilde B := \obj A \sqcup \obj B
\]
and
\begin{eqnarray*}
\tilde B(x,x'):= B(Fx,Fx)
&,&
\tilde B(y,x'):= B(y,Fx')
 , \\
\tilde B(x,y'):= B(Fx,y')
&, &
\tilde B(y,y'):= B(y,y')
\end{eqnarray*}
 for $x,x'\in \obj A$ and $y,y'\in \obj B$, with the involution and norm inherited from~$B$.
The $*$-functor $J$ is defined by $x\mapsto x$ on objects and by $F$ on arrows, and is clearly a cofibration, 
while~$Q$ maps $x\mapsto Fx$ ($x\in \obj A$) and $y \mapsto y$ ($y\in \obj B$), and is the identity on each Hom space.
Thus~$Q$ is fully faithful, and since it is evidently  surjective on objects it must be a trivial fibration (Cor.~\ref{cor:char_tr_fib}).
\end{proof}

This concludes the proof of Theorem~\ref{thm:first_model}.

\subsection{Cofibrant generation}
\label{sec:cof_gen}
To show that the unitary model is cofibrantly generated, we must find a small set of cofibrations $I$ (respectively,  of trivial cofibrations $J$) such that a $*$-functor is a trivial fibration (resp.\ a fibration) if and only if it has the right lifting property with respect to~$I$ (to~$J$). 
Moreover, we need to verify two smallness conditions for the domains of the maps in $I$ and~$J$ (see~\cite{hovey:model}*{Def.\ 2.1.17}). 

Note that the $*$-functor $0\colon \F \to \mathbf I$ is a trivial cofibration. 
Hence by Lemma~\ref{lemma:lifting_fib} we can already choose our set of generating trivial cofibrations to be $J:=\{0\colon \F \to \mathbf I \}$. 
Let us now construct the set~$I$.

Consider the unique $*$-functor $U\colon \emptyset \to \F$. 
Also, recall the universal $\Cstar$-category~$\mathbf 1$ generated by an arrow of norm one introduced in Example~\ref{ex:one_arrow}.  
Let $V\colon  \F \sqcup \F \to \mathbf 1$ be the $*$-functor sending the first copy of $\F$ to $0\in \obj \mathbf 1$ and the second one to $1 \in \obj \mathbf 1$.
Let $P$ be the pushout in $\Cstarcatun$ defined by the square 
\[
\xymatrix{
\F\sqcup \F \ar[r]^{V} \ar[d]_{V} &
\mathbf 1 \ar[d]
\\
\mathbf 1 \ar[r] & P
}
\]
and let $W\colon  P\to \mathbf 1$ be the $*$-functor with both components equal to $\id\colon  \mathbf 1\to \mathbf 1$ 
(equivalently, $P$ is the universal $\Cstar$-category generated by two parallel arrows $ 0\rightrightarrows 1 $ of norm at most one, and $W$ is determined by sending both of them to the generator $0\to 1$ of~$\mathbf 1$).

\begin{lemma} \label{lemma:liftings}
For every $*$-functor $F\colon A\to B$, the following hold true:
\begin{itemize}
\item[(i)]
$F$ has the right lifting property with respect to $U\colon \emptyset \to \F$ if and only if it is surjective on objects.

\item[(ii)]
$F$ has the right lifting property with respect to $V \colon  \F \sqcup \F \to \mathbf 1 $ if and only if it is full.

\item[(iii)]
$F$ has the right lifting property with respect to $W\colon  P \to \mathbf 1 $ if and only if it is faithful.
\end{itemize}
\end{lemma}

\begin{proof}
This is a straightforward translation of the three lifting properties, using both the universal properties of $\F$, $\mathbf 1$ and $P$ and the linearity of~$F$. 
\end{proof}

\begin{lemma} \label{lemma:lifting_tr_fib}
A $*$-functor is a trivial fibration if and only if it has the right lifting property with respect to the set of cofibrations 
$I:=\{U, V, W\} $.
\end{lemma}

\begin{proof}
Every trivial fibration has the right lifting property with respect to the maps in~$I$ -- which are clearly cofibrations -- by Lemma~\ref{lemma:MC4}\ (ii). Conversely, if $F\colon A\to B$ has the right lifting property with respect to $I$ \Changed{then} by Lemma~\ref{lemma:liftings} it is fully faithful and surjective on objects, and therefore it is a trivial fibration by Corollary~\ref{cor:char_tr_fib}.
\end{proof}

\begin{lemma}
\label{lemma:small}
The objects $\emptyset$, $\F$, and~$\F\sqcup\F$ are finite and~$P$ is $\aleph_1$-small in~$\Cstarcatun$, in the sense of 
\cite{hovey:model}*{Def.~2.1.3-4}.
\end{lemma}

\begin{proof}
Recall that the object-set of a colimit of $\Cstar$-categories is the colimit of the object-sets.
Thus if $A\in \{\emptyset, \F, \F\sqcup \F\}$, we see that every $*$-functor from~$A$ into a filtering colimit must factor through some  stage (in the case $A=\F\sqcup \F$, we need the colimit to be filtering to ensure that the two objects determining the $*$-functor out of~$\F\sqcup \F$ will eventually land in the same $\Cstar$-category). Thus in particular $\emptyset$,~$\F$ and~$\F\sqcup\F$ are finite objects of~$\Cstarcatun$. 
Now let us consider~$P$.  
Let $(B_0\to B_1\to \cdots B_\alpha \to B_{\alpha+1}\cdots )_{\alpha<\lambda}$ be a  sequence of $*$-functors indexed by an uncountable limit ordinal~$\lambda$, and let $F\colon  P\to \colim_{\alpha<\lambda}B_\alpha =:B$ be a $*$-functor into the colimit  of the sequence in~$\Cstarcatun$. 
Let $a,a'\colon 0\to 1$ be the two generating arrows of~$P$. 
By the construction of colimits as a completion (Prop.~\ref{prop:cocomplete}), the arrow~$Fa$ is the norm colimit in $B(F0,F1)$ of a countable sequence~$(b_n)_{n\in \N}$ with $b_n\in B_{\alpha_n}$; since $\lambda$ is an uncountable limit ordinal,~$Fa$ must be the image of some $b\in B_{\beta}$ for some $\beta < \lambda$. Similarly, there exists  a $b'\in B_{\beta'}$, for some $\beta'< \lambda$, mapping to~$Fa'\in B$. 
Moreover, reasoning with objects as above we see that~$b$ and~$b'$ must become parallel arrows at some stage $\gamma<\lambda$.
Thus there exists a $\gamma<\lambda$ and there exist arrows $f,f'\in B_\gamma(y_0,y_1)$ (for some $y_0,y_1\in \obj B_\gamma$) mapping to~$Fa$ and~$Fa'$ respectively. 
Therefore $F\colon P\to B$ factors through the unique $*$-functor $G:P\to B_\gamma$ determined by $Ga=f$ and $Ga'=f'$. This shows that $P$ is $\aleph_1$-small in~$\Cstarcatun$.
\end{proof}

By Lemma~\ref{lemma:small}, the domains of the generating maps~$I$ and~$J$ are small in~$\Cstarcatun$, thus ensuring that the (weaker) smallness condition required by the definition of cofibrantly generated model categories is satisfied.

Adding up all the lemmas, we conclude:

\begin{prop}\label{prop:cof_gen}
The unitary model structure on $\Cstarcatun$ is cofibrantly generated with the finite sets of generating cofibrations  and
trivial cofibrations
\[
I=\{ U\colon \emptyset \to \F , V\colon  \F \sqcup \F \to \mathbf 1, W \colon P\to \mathbf 1 \} 
\quad \textrm{ and } \quad
J= \{0\colon  \F \to \mathbf I \}
\]
defined above.
\qed
\end{prop}

\subsection{Compatibility with the maximal tensor product}

Next we verify that the unitary model and the maximal tensor product satisfy the axioms for a symmetric monoidal category, as in~\cite{hovey:model}*{Ch.\ 4}. Since the tensor unit object $\F$ is cofibrant, we need only prove the push-out product axiom, which is precisely the content of the next proposition.

\begin{prop} \label{prop:monoidal_model}
Let $F:A\to B$ and $F':A'\to B'$ be two $*$-functors and let
\[
F \Box F' 
\quad \colon  \quad 
(B \otimes_{\max} A' ) \sqcup_{A \otimes_{\max} A'} (A \otimes_{\max} B') \longrightarrow B \otimes_{\max} B'
\]
be the $*$-functor they induce by the pushout property. The following hold:
\begin{itemize}
\item[(i)]
If $F,F'\in \Cof$, then $F\Box F' \in \Cof$.
\item[(ii)]
If moreover $F\in \Weq$ or $F' \in \Weq$, then $F\Box F'\in \Weq$. 
\end{itemize}
\end{prop}

\begin{proof} 
The basic observation is that tensoring with any fixed object preserves cofibrations and weak equivalences; this is immediately verified by direct inspection.

For any two $\Cstar$-categories $M,N$ we have $\obj (M\otimes_{\max} N)= \obj M \times \obj N$; similarly, the underlying sets of objects in a pushout of $\Cstarcatun$ form a pushout of sets. With this in mind, it is easy to verify that the injectivity of $\obj F$ and $\obj F'$ implies the injectivity of $\obj (F\Box F')$, thus proving~(i).
\[
\xymatrix{
A\otimes A' \ar[dd]_{F\otimes A'} \ar[rr]^-{A\otimes F'} &&  A \otimes B' \ar[dd]^{F\otimes B'} \ar[dl]_{G} \\
& (B \otimes A') \sqcup_{A\otimes A' } (A \otimes B' ) \ar@{..>}[dr]^{F\Box F'} & \\
B \otimes A' \ar[ur]^-{} \ar[rr]^-{B\otimes F'} && B \otimes B'
}
\]
Now assume that $F$ is a trivial cofibration. Then (by the first remark) $F \otimes_{\max}A' $ is also a trivial cofibration. Therefore its pushout $G$ along $A\otimes F'$ is again a trivial cofibration (\cite{hovey:model}*{Cor.\ 1.1.11}), in particular $G\in \Weq$. Since $F\otimes B'\in \Weq$ too, we conclude that $F\Box F'\in \Weq$ by the 2-out-of-3 property of weak equivalences.  
\end{proof}

\subsection{Simplicial structure}

Using both the internal Hom $\Cstar$-categories $\Cstar(A,B)$ and 
the maximal groupoid $\Cstar$-categories $\Cstarmax (\mathcal G)$ of Section~\ref{subsec:max_gpd},  we now define a simplicial structure on $\Cstarcatun$ compatible with the unitary model. 

We start by defining an adjunction between $\Cstarcatun$  and the category $\sSet$ of simplicial sets. Let $\pi\colon  \sSet \to \Cstarcatun$ be the functor assigning to a simplicial set $K$ the maximal groupoid $\Cstar$-category of its fundamental groupoid $\Pi(K)$. We recall that $\Pi(K)$ has objects
$\obj \Pi(K) := K_0$, and its arrows are generated by isomorphisms $ k\colon  d_1k \to d_0k $ (for $k \in K_1$) with relations $d_0\ell \cdot d_2\ell = d_1\ell$ (for $\ell \in K_2$). 
(See \emph{e.g.},~\cite{goerss-jardine}*{Ch.\ III.1}, where our $\Pi(K)$ is the construction denoted $GP_*(K)$.)

Let $\nu\colon  \Cstarcatun \to \sSet$ be the functor mapping a $\Cstar$-category $A$ to the simplicial nerve $N(\unitaries A)$ of its unitary groupoid.

\begin{lemma} \label{lemma:adjunction}
The functor $\pi\colon  \sSet \to \Cstarcatun $ is left adjoint to $\nu\colon  \Cstarcatun \to \sSet $.
\end{lemma}

\begin{proof}
Since by definition $\pi = \Cstarmax \cdot \Pi$ and $\nu = N \cdot \unitaries$, it suffices to compose the adjunction
$\Cstarmax \colon  \Gpd \leftrightarrows \Cstarcatun : \!\unitaries$ of Proposition~\ref{prop:Cstar_adjoint} with the well-known adjunction $\Pi\colon  \sSet \leftrightarrows  \Gpd :\! N$ between the fundamental groupoid and the simplicial nerve, see~\cite{goerss-jardine}*{p.\ 154}. 
\end{proof}

\begin{prop} \label{prop:Quillen_adjunction}
The adjunction $\pi\colon  \sSet \leftrightarrows \Cstarcatun : \!\nu $ is a Quillen pair.
\end{prop}

\begin{proof}
It suffices to verify that $\pi$ preserves cofibrations and that it sends the generating trivial cofibrations of $\sSet$ to trivial cofibrations. If $f\colon K\to L$ is a cofibration (\emph{i.e.}, a dimensionwise injective map), then in particular it is injective in dimension zero and therefore $\pi(f)$ is injective on objects, \emph{i.e.}, it is a cofibration of~$\Cstarcatun$.  

Now let $\iota_{n,k} \colon  \Lambda^k[n] \to \Delta[n]$  (for $n\geq 1, 0\leq k\leq n$) be a generating trivial cofibration of $\sSet$. 
Precisely as in~\cite{rezk:folk}*{Thm.\ 6.1}, one can verify that, for $n> 1$, the functor $\Pi(\iota_{n,k})$ is an isomorphism of fundamental groupoids, and that $\Pi(\iota_{1,k})$ is the inclusion $1 \to I$ of the trivial group as one end of the interval groupoid (Def.~\ref{defi:interval}). 
Hence each $\pi (\iota_{n,k})= \Cstarmax \Pi(\iota_{n,k})$ is an isomorphism of $\Cstar$-categories for  $n>1$, and $\pi(\iota_{1,k})=(0\colon \F \to \mathbf I)$ is the generating trivial cofibration of $\Cstarcatun$. 
In particular, these are all  trivial cofibrations as required.
\end{proof}

\begin{definition}[Simplicial structure]
\label{defi:simpl_struct}
Define three functors
\begin{eqnarray*}
A, K \mapsto A \otimes K &\colon & \Cstarcatun \times \sSet \to \Cstarcatun \\
A, K \mapsto A^K  &\colon &  \Cstarcatun \times \sSet^\op \to \Cstarcatun \\
A,B \mapsto \mathrm{Map}(A,B) &\colon & \Cstarcatun^\op \times \Cstarcatun \to \sSet
\end{eqnarray*}
by setting
\[
A\otimes K := A\otimes_{\max} \pi K
\quad,\quad
A^K := \Cstar(\pi K, A)
\quad\textrm{ and }\quad
\mathrm{Map}(A,B):= \nu \Cstar(A,B).
\]
\end{definition}

\begin{thm} \label{thm:simpl_struct}
The operations of Definition~\ref{defi:simpl_struct} and the unitary model structure of Theorem~\ref{thm:first_model} make $\Cstarcatun$ into a simplicial model category.
\end{thm}

\begin{proof}
The required natural isomorphisms
\[
\Cstarcatun (A \otimes K , B) \simeq
 \Cstarcatun (A , B^K) \simeq 
 \sSet (K, \mathrm{Map} (A,B))
\]
are obtained by combining the various adjunctions we have constructed so far.
The required axiom involving the push-out product follows from Proposition~\ref{prop:Quillen_adjunction} and Proposition~\ref{prop:monoidal_model}.
\end{proof}

The latter theorem is also a formal consequence of the next one.

\begin{thm} \label{thm:algebra}
The functor $\pi\colon \sSet \to \Cstarcatun$ is a simplicial symmetric monoidal left Quillen functor. Thus, in the language of~\cite{hovey:model}, the unitary model category $\Cstarcatun$ is a symmetric $\sSet$-algebra.
\end{thm}

\begin{proof}
The functor $\Pi\colon \sSet \to \Gpd$ commutes with finite products, thus defining a symmetric monoidal functor $(\sSet, \times , \pt)\to (\Gpd, \times, 1)$; by Theorem~\ref{thm:max_products}, $\Cstarmax$ is also symmetric monoidal, and therefore so is their composition~$\pi$. We have seen (Prop.~\ref{prop:Quillen_adjunction}) that $\pi$ is a left Quillen functor, and by Definition~\ref{defi:simpl_struct} it transports the simplicial structure of $\sSet$ to that of~$\Cstarcatun$.  
\end{proof}

\subsection{A stroll in the Quillen neighborhood}

There is a commutative tetrahedron of model categories where every edge is a Quillen adjoint pair:
\begin{equation*}
\xymatrix{
& \sSet 
 \ar[dl]_-{(\Pi, N\, iso)} 
  \ar[dd]|-<<<<<{(\Pi, N)}|{\phantom{M}} 
    \ar[dr]^-{(\pi, \nu)}
 & \\
\Cat
 \ar[rr]
  &
   & \Cstarcatun  \\
& \Gpd
 \ar[ul]^-{ ( inc, iso )} 
  \ar[ur]_-{(\Cstarmax, \unitaries)}
 &
}
\end{equation*}
(for each pair $(L,R)$, the arrow shows the direction of the left adjoint~$L$).

The triangle on the left hand side is classical and comprises the usual model category of simplicial sets, together with the so-called \emph{canonical}, or \emph{folk}, models on categories and groupoids. For both $\Cat$ and~$\Gpd$, the weak equivalences are the equivalences in the usual sense of category theory, the cofibrations are the functors that are injective on objects, and the fibrations are the functors $F\colon C\to D$ enjoying the lifting property for isomorphisms: if $u\colon Fx \simeq y$ is an isomorphism in~$D$, then there exists an isomorphism $v$ in $C$ such that $Fv=u$ (\emph{cf.}~\cite{rezk:folk} and~\cite{anderson}*{$\S$5}).

It should now be obvious that the unitary model for $\Cstarcatun$ has been adapted from the latter ones.
All labeled Quillen pairs on the diagram should be clear. Just recall that $\Pi$ denotes the fundamental groupoid, left adjoint to the simplicial nerve~$N$. With $inc$ we have denoted the inclusion  of groupoids into categories, whose right adjoint $iso$ maps a category to its subcategory of isomorphisms.
Observe that all model categories are symmetric monoidal ($\Cstarcatun$ for~$\otimes_{\max}$, the other three for the categorical product), and that all left Quillen functors are symmetric monoidal.

We now explain the remaining, still unlabeled, Quillen pair $\Cat \leftrightarrows \Cstarcatun$, which provides a direct connection between categories and $\Cstar$-categories.
\begin{definition}[$\ism$ and $\Cstarism$]
\label{defi:ism}
Recall that an arrow $a\in A$ in a $\Cstar$-category is an \emph{isometry} if $a^*a=1$. The isometries in $A$ form a subcategory $\ism A$, and since $*$-functors preserve isometries we have an induced functor $\ism\colon  \Cstarcatun\to \Cat$. 

In the other direction, for any small category~$C$ let $\Cstarism ( C)$ be the universal $\Cstar$-category realizing all arrows of $C$ as isometries; 
in other words, $\Cstarism ( C)$ is the universal $\Cstar$-category for the underlying quiver of~$C$, with relations the composition in $C$ together with $\{c^*c=1_x\mid c\in C(x,y), x,y\in \obj C\}$ (this is easily seen to be admissible as in Section~\ref{sec:univ} because, like unitaries,  isometries have norm one or zero). 
Let $\Cstarism\colon  \Cat \to \Cstarcatun$ denote the resulting functor.
\end{definition}

\begin{prop}
\label{prop:ism}
The two functors in  Definition~\ref{defi:ism} form a Quillen adjoint pair, $\Cstarism\colon \Cat \leftrightarrows \Cstarcatun : \!\ism$, making the above tetrahedron commute \textup(to have strict commutativity, we should really redefine $\Cstarmax:= \Cstarism \cdot  inc$\textup). 
Moreover, $\Cstarism$ is a symmetric monoidal functor $(\Cat,\times,1)\to (\Cstarcatun, \otimes_{\max},\F)$.
\end{prop}

\begin{proof}
It is possible to recycle almost \emph{verbatim} the proofs of the analogous facts for~$\Cstarmax$ in Section~\ref{sec:univ}; only the proof of the analog of Lemma~\ref{lemma:UP_max} requires a little (straightforward) adjustment. 
We leave it to the interested reader. 
The commutativity of the tetrahedron is immediate from the definitions.
\end{proof}


\end{document}